\documentclass[11pt,reqno,oneside]{amsart}

\usepackage{graphicx}
\usepackage{amssymb}
\usepackage{epstopdf}
\usepackage{cases}
\usepackage[abbrev,nobysame]{amsrefs}
\usepackage[pagewise]{lineno}
\usepackage[abbrev,nobysame]{amsrefs} \AtBeginDocument{\def\MR#1{}} 

\usepackage[a4paper, total={6in, 9in}]{geometry}

\usepackage{amsmath,amsfonts,amsthm,mathrsfs,amssymb, cite}

\usepackage[usenames]{color}

\usepackage{hyperref}
\usepackage{xcolor}
\hypersetup{
  colorlinks,
  linkcolor={blue!80!black},
  urlcolor={blue!80!black},
  citecolor={blue!80!black}
}
\usepackage[capitalize,noabbrev]{cleveref}

\newtheorem{thm}{Theorem}[section]

\newtheorem{lem}[thm]{Lemma}
\newtheorem{prop}[thm]{Proposition}
\theoremstyle{definition}
\newtheorem{defn}{Definition}[section]

\theoremstyle{remark}

\newtheorem{rem}{Remark}[section]
\numberwithin{equation}{section}

\usepackage{cite}
\usepackage{mathrsfs} 
\allowdisplaybreaks

\newcommand{\R}{\mathbb{R}} \newcommand{\mathR}{\mathbb{R}}
\newcommand{\Rn}{{\mathR^n}}

\newcommand{\calF}{\mathcal{F}}
\newcommand{\calR}{\mathcal{R}}

\newcommand{\scrS}{\mathscr{S}}

\newcommand{\supp}{\mathop{\rm supp}}

\newcommand{\diam}{\mathop{\rm diam}}
\newcommand{\divr}{\mathop{\rm div}}

\newcommand{\Forall}{{~\forall\,}}
\newcommand{\Exists}{{~\exists\,}}
\newcommand{\st}{\textrm{~s.t.~}}
\newcommand{\aee}{\textrm{~a.e.~}}
\newcommand{\ass}{\textrm{~a.s.~}}

\newcommand{\comment}[1]{}
\newcommand{\dif}[1]{\,\mathrm{d}{#1}} 
\newcommand{\nrm}[2][]{ \| {#2} \|_{#1}} 
\newcommand{\Nrm}[2][]{ \big\| {#2} \big\|_{#1}} 
\newcommand{\agl}[1][\cdot]{ \langle {#1} \rangle} 
\newcommand{\Agl}[1][\cdot]{ \big\langle {#1} \big\rangle} 

\usepackage{tikz}
\newcommand{\Rk}{{\mathcal{R}_{k}}}

\title[Inverse problem for random Schr\"odinger equation]{Inverse problem for a random Schr\"odinger equation with unknown source and potential}

\author{Hongyu Liu}
\address{Department of Mathematics, City University of Hong Kong, Hong Kong SAR, China}
\email{hongyu.liuip@gmail.com, hongyliu@cityu.edu.hk}

\author{Shiqi Ma}
\address{School of Mathematics, Jilin University, Changchun, China}
\email{mashiqi@jlu.edu.cn, mashiqi01@gmail.com}

\begin{document}

\begin{abstract}

	In this paper, we study an inverse scattering problem associated with the time-harmonic Schr\"odinger equation where both the potential and the source terms are unknown. The source term is assumed to be a generalised Gaussian random distribution of the microlocally isotropic type, whereas the potential function is assumed to be deterministic. The well-posedness of the forward scattering problem is first established in a proper sense. It is then proved that the rough strength of the random source can be uniquely recovered, independent of the unknown potential, by a single realisation of the passive scattering measurement. In addition to the use of a single sample of the passive measurement for two unknowns, another significant feature of our result is that there is no geometric restriction on the supports of the source and the potential: they can be separated, or overlapped, or one containing the other. 
	
%
%

	\medskip

	
	
\end{abstract}

\maketitle

\section{Introduction} \label{sec:intro-MLSchroEqu2018}

\subsection{Statement of the main results}

We are mainly concerned with the quantum scattering problem governed by the following time-harmonic Schr\"odinger equation (cf.~\cite{eskin2011lectures,griffiths2016introduction})
\begin{subequations} \label{eq:1-MLSchroEqu2018}
	\begin{numcases}{}
	\displaystyle{ (-\Delta-E+V(x)) u(x, \sqrt{E}, \omega)= f(x,\omega), \quad x\in\mathbb{R}^3, } \label{eq:1a-MLSchroEqu2018} \medskip\\
	\displaystyle{ \lim_{r\rightarrow\infty} r\left(\frac{\partial u}{\partial r}-\mathrm{i}\sqrt{E} u \right)=0,\quad r:=|x|. } \label{eq:1c-MLSchroEqu2018}
	\end{numcases}
\end{subequations}
In \eqref{eq:1a-MLSchroEqu2018}--\eqref{eq:1c-MLSchroEqu2018}, $u$ is the scattered wave field generated by the interaction of the source $f$ and the scattering potential $V$, and $E\in\mathbb{R}_+$ signifies the energy level. We write $k:=\sqrt{E}$, namely $E=k^2$, which can be regarded as the wavenumber for the time-harmonic wave scattering. $\omega$ in \eqref{eq:1a-MLSchroEqu2018} is the random sample belonging to $\Omega$ with $(\Omega,\mathcal F, \mathbb P)$ being a complete probability space.
The limit \eqref{eq:1c-MLSchroEqu2018} is known as the Sommerfeld radiation condition (SRC) (cf.~\cite{colton2012inverse}), which holds uniformly in the angular variable $\hat x:=x/|x|\in\mathbb{S}^2$  that characterizes the outgoing nature of the scattered wave field $u$.

In our study, $V$ is assumed to be a deterministic smooth function,
and $f$ is assumed to be a compactly supported generalised Gaussian random distribution of the microlocally isotropic type (cf.~\cite{caro2016inverse,Lassas2008}), which is rigorously characterised as follows for the self-containedness of our study.
It means that $f(\cdot,\omega)$ is a random distribution, and the mapping 
\begin{equation*}
\omega \in \Omega \ \mapsto \ \agl[f(\cdot,\omega),\varphi] \in \mathbb C,\quad \varphi\in\scrS(\Rn),
\end{equation*}
is a Gaussian random variable whose probabilistic measure depends on the test function $\varphi$. Here and also in what follows, $\scrS(\Rn)$ stands for the Schwartz space.
From this explanation we see the convolution $f(\cdot,\omega) * \varphi(x)$ is measurable with respect to the canonical product measure on the product space $\R^3 \times \Omega$.
Since both $\agl[f(\cdot,\omega),\varphi]$ and $\agl[f(\cdot,\omega),\psi]$ are random variables for $\varphi$, $\psi \in \scrS(\Rn)$,
from a statistical point of view, the covariance between these two random variables,
\begin{equation} \label{eq:CovDef-MLSchroEqu2018}
\mathbb E_\omega \big( \agl[\overline{f(\cdot,\omega) - \mathbb{E} (f(\cdot,\omega))},\varphi] \agl[f(\cdot,\omega) - \mathbb{E} (f(\cdot,\omega)),\psi] \big),
\end{equation}
can be understood as the covariance of $f$, where $\mathbb E_\omega$ means to take expectation on the argument $\omega$. 
Formula \eqref{eq:CovDef-MLSchroEqu2018} defines an operator $\mathfrak C_f$,
\[
\mathfrak C_f \colon \varphi \in \scrS(\Rn) \ \mapsto \ \mathfrak{C}_f (\varphi) \in \scrS'(\Rn),
\]
in a way that
\(\mathfrak C_f (\varphi) \colon \psi \in \scrS(\Rn) \ \mapsto \ (\mathfrak C_f (\varphi))(\psi) \in \mathbb C\)
where
\[
(\mathfrak C_f (\varphi))(\psi) := \mathbb E_\omega \big( \agl[\overline{f(\cdot,\omega) - \mathbb{E} (f(\cdot,\omega))},\varphi] \agl[f(\cdot,\omega) - \mathbb{E} (f(\cdot,\omega)),\psi] \big).
\]
The operator $\mathfrak{C}_f$ is called the covariance operator of $f$.

\begin{defn} \label{defn:migr-MLSchroEqu2018}
	A generalized Gaussian random distribution $f$ on $\R^3$ is called microlocally isotropic with rough order $-m$ and rough strength $\mu(x)$ in a bounded domain $D$, if the following conditions hold:
	\begin{enumerate}
		\item the expectation $\mathbb E (f)$ is in $\mathcal{C}_c^\infty(\R^3)$ with $\supp \mathbb E (f) \subset D$;
		
		\item $f$ is supported in $D$ a.s. (namely, almost surely);
		
		\item the covariance operator $\mathfrak C_f$ is a classical pseudodifferential operator of order $-m$;
		
		\item $\mathfrak C_f$ has a principal symbol of the form $\mu(x)|\xi|^{-m}$ with $\mu \in \mathcal{C}_c^\infty(\R^3;\R)$, $\supp \mu \subset D$ and $\mu(x) \geq 0$ for all $x \in \R^3$.
	\end{enumerate}
\end{defn}

Though Definition \ref{defn:migr-MLSchroEqu2018} is given in three dimensions, it can be easily generalize to any dimensions.
In fact, we can provide two illustrative examples on microlocally isotropic Gaussian random distributions as follows. 
The first one is
\(
f(x,\omega)
= \sqrt{\mu(x)} (I - \Delta)^{-m/4} W(x,\omega),
\)
where $W$ signifies a three-dimensional spacial Gaussian white noise \cite[Example 2.2]{caro2016inverse}.
$W$ satisfies $\mathbb E_\omega \big( G(x,\omega) \big) = 0$ and $\mathbb E_\omega \big( W(x,\omega) G(y,\omega) \big) = \delta(x-y)$.
In two dimensions, a sample of $W(x,\omega)$, i.e., $x \mapsto G(x,\omega_0)$ for a certain fixed $\omega_0$, can be seen as a very rough surface.
The pseudodifferential operator $(I - \Delta)^{-m/4}$ acts as a mollifier to smooth out this rough surface, and the smoothing effect can be tuned through the parameter $m$.
The second example is
\(
f(x,\omega)
= \sqrt{\mu(x)} X_H(x,\omega),
\)
where $X_H \colon \R^3 \times \Omega \to \R$ signifies a multidimensional fractional Brownian motion in $\R^3$ with Hurst index $H \in (0,1)$ \cite[Example 2.3]{caro2016inverse}.
The Hurst exponent describes the raggedness of the resultant motion.
A higher value of $H$ leads to a smoother motion.
For a microlocally isotropic Gaussian random distribution $f$, the $\mu(x)$ induced by the covariance operator $\mathfrak C_f$ has a role similar to the covariance function of $f$.
We shall provide more explanations in what follows in Section \ref{subsec:Prel-MLSchroEqu2018}.

In what follows, we abbreviate a microlocally isotropic Gaussian random distribution as an {\it m.i.g.r.}~function.
Let $f$ be an m.i.g.r.~function. We consider the corresponding forward and inverse scattering problems associated with the Schr\"odinger equation \eqref{eq:1a-MLSchroEqu2018}--\eqref{eq:1c-MLSchroEqu2018}. For the forward scattering problem, we shall show that there exists a well-defined scattering map in a proper sense as follows:
\begin{equation*} 
(f, V) \ \rightarrow \ u^\infty(\hat x, k, \omega),\ \ \hat x \in \mathbb{S}^2,\, k\in\mathbb{R}_+,\, \omega \in \Omega,
\end{equation*}
where $u^\infty$ is a random distribution on the unit sphere and is called the far-field pattern. That is, for a given pair of $(f, V)$, by solving the forward scattering system \eqref{eq:1a-MLSchroEqu2018}--\eqref{eq:1c-MLSchroEqu2018}, one can obtain the far-field pattern in a proper sense. It is noted that the far-field pattern is generated through the interaction of the source $f$ and the scattering potential $V$, and hence it carries the information of $f$ and $V$. The inverse scattering problem is concerned with recovering the unknown $f$ or/and $V$ by knowledge of the far-field pattern, namely,
\begin{equation}\label{eq:ip-MLSchroEqu2018}
\big\{u^\infty(\hat x, k, \omega) \,;\, \hat x \in \mathbb{S}^2, k\in\mathbb{R}_+,\, \omega \in \Omega \big\} \ \rightarrow \ (f, V). 
\end{equation}
It is noted that the measured far-field pattern in \eqref{eq:ip-MLSchroEqu2018} is produced by the unknown source, and it is referred to as the passive measurement in the scattering theory. 
This is in difference to the active measurement, where one exerts certain known wave sources to generate the scattered waves in order to recover the unknown objects.

In what follows, we shall impose the following mild regularity assumption on the potential $V$:
\begin{equation} \label{eq:asp-MLSchroEqu2018}
V \in C^{5}(\R^3),\, V \in L_{3/2+\epsilon}^2(\R^3) \text{~and~} \partial^\alpha V \in L_{1/2+\epsilon}^{\infty}(\R^3) \text{~for~} \forall \alpha \colon |\alpha| \leq 2.
\end{equation}	
It is emphasized that the above $C^5$-regularity requirement is mainly a technical condition, which shall be needed in our subsequent stationary phase argument (cf. \eqref{eq:g10-MLSchroEqu2018}).
For the inverse scattering problem, we shall prove
\begin{thm} \label{thm:Unisigma-MLSchroEqu2018}
	Let $f$ be an m.i.g.r.~distribution  such that $\supp(f)\subset D_f$ where $D_f$ is a bounded domain in $\mathbb{R}^3$ and $V$ satisfies \eqref{eq:asp-MLSchroEqu2018}.
	Let $\mu$ be the rough strength of $f$. Suppose that $f$ is of order $-m$ with 
	$2 < m < 3$.
	Then the far-field data $u^\infty(\hat x, k, \omega)$ for all $(\hat x, k)\in\mathbb{S}^2\times\mathbb{R}_+$ and a fixed $\omega\in\Omega$ can uniquely recover $\mu$ almost surely, independent of $V$.
	Moreover, a reconstruction formula is proposed in \eqref{eq:sigmaHatRecSingle-MLSchroEqu2018}.
\end{thm}

\begin{rem}
	Theorem~\ref{thm:Unisigma-MLSchroEqu2018} indicates that a single realisation of the passive scattering measurement can uniquely recover the rough strength of the unknown source, independent of the scattering potential and the expectation of the source.
	In fact, our arguments in what follows in proving the theorem actually yield an explicit formula in recovering $\mu$ by the given far-field data (cf.~formula \eqref{eq:sigmaHatRecSingle-MLSchroEqu2018}).
	It is emphasized we do not assume $V$ to be compactly supported.
	This is in sharp difference to our earlier study \cite{HLSM2019both}, where $V$ was also assumed to be compactly supported, and $\supp V$ and $\supp f$ are assumed to be well separated in the sense that their convex hulls stay a positive distance away from each other. We shall discuss more about this point in Section 1.2. 
\end{rem}

\begin{rem}\label{rem:1.2}
	In Theorem~\ref{thm:Unisigma-MLSchroEqu2018}, we only consider the recovery of the rough strength of the source, which is independent of the expectation of the source and the scattering potential, both of them being unknown. It is pointed out that in essence one can also recover the expectation of the source, but would need to make use of the full-realisation of the passive scattering measurement. Moreover, if active scattering measurement is further used, one may also be able to recover the potential by following similar arguments as in \cite{HLSM2019both}. However, in our view, the result presented in Theorem~\ref{thm:Unisigma-MLSchroEqu2018} is the most significant advancement in understanding the inverse scattering problem associated with the random Schr\"odinger equation \eqref{eq:1a-MLSchroEqu2018}--\eqref{eq:1c-MLSchroEqu2018}. 
\end{rem}

\subsection{Discussion of our results and literature review}

Inverse scattering theory is a central topic in the mathematical study of inverse problems and on the other hand, it is the fundamental basis for many industrial and engineering applications, including radar/sonar, geophysical exploration and medical imaging. 
It is concerned with the recovery of unknown/inaccessible scattering objects by knowledge of the associated wave scattering measurements away from the objects. 
The scattering object could be a passive inhomogeneous medium or an active source. 
The scattering measurement might be generated by the underlying unknown source, referred to as the passive measurement, or by exerting a certain known wave field, referred to as the active measurement. 
Both the inverse medium scattering problem and the inverse source scattering problem in the deterministic settings have been intensively and extensively investigated in the literature; 
see e.g. \cite{Bsource, BL2018, ClaKli, Deng2019Onan, Klibanov2013, Zhang2015, WangGuo17, colton2012inverse, uhlmann2013inverse} for some recent related studies and the references cited therein.
The simultaneous recovery of an unknown source as well as the material parameter of an inhomogeneous medium by the associated passive measurement was considered \cite{KM1,liu2015determining}, which arises in the photoacoustic and thermoacoustic tomography as an emerging medical imaging modality. 
Similar inverse problems were also considered in \cite{Deng2019onident, Deng2020onident} associated with the magnetohydrodynamical system and in \cite{Deng2019Onan} associated with the Maxwell system that are related to the geomagnetic anomaly detection and the brain imaging, respectively.
Inverse scattering problems in the random settings have also received considerable attentions in the literature; see e.g. \cite{LassasA, Lassas2008, blomgren2002, borcea2002, borcea2006,
	bao2016inverse, Lu1, Yuan1, LiHelinLiinverse2018, HLSM2019determining, HLSM2019both} and the references therein.
In \cite{SM2020onrecent}, the second author of the present paper gives a review on recent progress of single-realization recoveries of random Schr\"odinger systems, and discuss some key ideas in \cite{Lassas2008} and \cite{HLSM2019both}.

Among the aforementioned studies of the random inverse problems, we are particularly interested in the case where a \emph{single} random sample is used to recover the unknowns.
Papanicolaou \cite{blomgren2002, borcea2002, borcea2006} studied the single realization recoveries that are more engineering-oriented.
In \cite{LassasA, Lassas2008}, Lassas et.~al.~considered the inverse scattering problem for the two-dimensional random Schr\"odinger system, and recovered the rough strength of the potential by using the near-field data under a single random sample.
In \cite{HLSM2019determining, HLSM2019both}, we studied the random Schr\"odinger system in a different setting and recovered the rough strength under a single random sample. In \cite{LiHelinLiinverse2018}, Li et.~al.~considered the inverse scattering problem of recovering a random source under a single random sample. It is emphasized that the recovery of the potential is comparably more challenging than the recovery of the source. In this paper, we shall consider the case that both the source and the potential are unknown, making the corresponding study radically more challenging.

Recently, the m.i.g.r.~model has been under an intensive study; see \cite{Lassas2008,LassasA,caro2016inverse,LiHelinLiinverse2018,LiLiinverse2018} and the references cited therein.
Two important parameters of the m.i.g.r.~distribution are its rough order and {rough strength}. Roughly speaking, the rough order determines the degree of spatial roughness of the m.i.g.r., and the rough strength indicates its spatial correlation length and intensity.
The rough strength also captures the micro-structure of the object in interest \cite{Lassas2008}.

The current article is a continuation of our study in two recent works \cite{HLSM2019determining,HLSM2019both} on the inverse scattering problem \eqref{eq:ip-MLSchroEqu2018} associated with the Schr\"odinger system \eqref{eq:1a-MLSchroEqu2018}- \eqref{eq:1c-MLSchroEqu2018}. The major connections and differences among those studies can be summarised as follows.
\begin{enumerate}
	\item In \cite{HLSM2019determining}, we considered the case that the random part of the source is a spatial Gaussian white noise, whereas the potential term is deterministic. It is proved that a single realisation of the passive scattering measurement can uniquely recover the variance of the random source, independent of the potential. However, in this paper, we derive a similar unique recovery result, but for the random source being a much more general m.i.g.r. distribution. As shall be seen in our subsequent analysis, the m.i.g.r source makes the corresponding analysis radically much more challenging.
	
	\item In \cite{HLSM2019both}, both the source and potential terms were assumed to be random of the m.i.g.r.~type. It was proved that a single realisation of the passive scattering measurement can uniquely recover the rough strength of the source, independent of the potential. However, in order to achieve such a unique recovery result, a restrictive geometric condition is critically required that the convex hulls of the supports of the source and potential are well separated. In this paper, we completely remove this geometric condition without imposing any assumption on the bounded supports of the source and the potential. As shall be seen in our subsequent study, the removal of this geometric condition makes the relevant analysis much more challenging and technical, and we develop novel mathematical techniques to handle this general geometric situation. On the other hand, it is remarked that the cost of removing this restrictive geometric condition is that we need to require the unknown potential to be deterministic. According to our intricate and subtle estimates in establishing the determination results in \cite{HLSM2019both} and Theorem~\ref{thm:Unisigma-MLSchroEqu2018} in the present paper, we believe that such a cost is unobjectionable. 
	
	\item In both \cite{HLSM2019determining} and \cite{HLSM2019both}, it was shown that if full scattering measurement is used, namely both passive and active measurements are used, then both the source and the potential can be recovered. In this paper, we only consider the recovery of the source by using the associated passive measurement. Nevertheless, it is remarked that if full measurement is used, then one can also establish the recovery of both the source and the potential by following similar arguments to those in \cite{HLSM2019determining} and \cite{HLSM2019both}; see Remark~\ref{rem:1.2} as well. 
	
\end{enumerate}

The rest of the paper is organized as follows. In Section \ref{sec:2-MLSchroEqu2018}, we present the well-posedness of the direct scattering problem. 
Section \ref{sec:AsyEst-MLSchroEqu2018} establishes several critical asymptotic estimates.
Finally in Section \ref{sec:VarRec-MLSchroEqu2018} we prove the unique recovery of the rough strength of the random source.

\section{Well-posedness of the direct problem} \label{sec:2-MLSchroEqu2018}

In this section, the unique existence of a {\it distributional solution} shall be established to the random Schr\"odinger system \eqref{eq:1-MLSchroEqu2018}.

We first fix some notations that shall be used throughout the rest of the paper. We write $\mathcal{L}(\mathcal A, \mathcal B)$ to denote the set of all the bounded linear mappings from a normed vector space $\mathcal A$ to a normed vector space $\mathcal B$. 
For any mapping $\mathcal K \in \mathcal{L}(\mathcal A, \mathcal B)$, we denote its operator norm as $\nrm[\mathcal{L}(\mathcal A, \mathcal B)]{\mathcal K}$. 
We use $C$ and its variants, such as $C_D$, $C_{D,f}$, to denote some generic constants whose particular values may change line by line. 
For two quantities $\mathbf{P}$ and $\mathbf{Q}$, we write $\mathbf{P}\lesssim \mathbf{Q}$ to signify $\mathbf{P}\leq C \mathbf{Q}$ and $\mathbf{P} \simeq \mathbf{Q}$ to signify $\widetilde{C}\mathbf{Q}\leq \mathbf{P} \leq C \mathbf{Q}$, for some generic positive constants $C$ and $\widetilde{C}$.
We may write ``almost everywhere'' as~``a.e.''~and ``almost surely'' as~``a.s.''~for short. 
We use $|\mathcal S|$ to denote the Lebesgue measure of any Lebesgue-measurable set $\mathcal S$.
In dimension $n \geq 1$, the Fourier transform and its inverse of a function $\varphi$ are defined respectively as
\begin{align*} 
	& \calF \varphi (\xi) = \widehat \varphi(\xi) := (2\pi)^{-n/2} \int_\Rn e^{-{\textrm{i}}x\cdot \xi} \varphi(x) \dif x, \\
	& \calF^{-1} \varphi (\xi) := (2\pi)^{-n/2} \int_\Rn e^{{\textrm{i}}x\cdot \xi} \varphi(x) \dif x.
\end{align*}
Write $\agl[x] := (1+|x|^2)^{1/2}$ for $x \in \Rn$, $n\geq 1$. We introduce the following weighted $L^p$-norm and the corresponding function space over $\Rn$ for any $\delta \in \R$,
\begin{equation} \label{eq:WetdSpace-MLSchroEqu2018}
	\begin{aligned}
		\nrm[L_\delta^p(\Rn)]{\varphi} := \ & \nrm[L^p(\Rn)]{\agl[\cdot]^{\delta} \varphi(\cdot)} = \big( \int_{\Rn} \agl[x]^{p\delta} |\varphi|^p \dif{x} \big)^{\frac 1 p}, \\
		L_\delta^p(\Rn) := \ & \{\, \varphi \in L_{loc}^1(\Rn) \,;\, \nrm[L_\delta^p(\Rn)]{\varphi} < +\infty \,\}.
	\end{aligned}
\end{equation}
We also define $L_\delta^p(S)$ for any subset $S$ in $\Rn$ by replacing $\Rn$ in \eqref{eq:WetdSpace-MLSchroEqu2018} with $S$. 
In what follows, we may write $L_\delta^2(\R^3)$ as $L_\delta^2$ for short without ambiguities.

Next, we present some basics about the random model and some other preliminaries for the subsequent use.

\subsection{Random model and preliminaries} \label{subsec:Prel-MLSchroEqu2018}

The following lemma shows the precise relationship between the regularity of $h$ and its rough order.

\begin{lem} \label{lem:migrRegu-MLSchroEqu2018}
	Let $f$ be a m.i.g.r.~of rough order $-m$ in $D_f$ in dimension $n \geq 1$. Then $f \in H^{s,p}(\Rn)$ almost surely for any $1 < p < +\infty$ and $s < (m-n)/2$.
\end{lem}
\begin{proof}
	See \cite[Proposition 2.4]{caro2016inverse}.
\end{proof}

\smallskip

Let $f$ be as in Lemma \ref{lem:migrRegu-MLSchroEqu2018} in dimension $n \geq 1$.
By the Schwartz kernel theorem \cite[Theorem 5.2.1]{hormander1985analysisI}, there exists a kernel $K_f(x,y)$ with $\supp K_f \subset D_f \times D_f$ such that
\begin{align} 
(\mathfrak C_f \varphi)(\psi) 
& = \mathbb E_\omega ( \agl[\overline{f(\cdot,\omega) - \mathbb{E} (f(\cdot,\omega))},\varphi] \agl[f(\cdot,\omega) - \mathbb{E} (f(\cdot,\omega)),\psi]) \nonumber\\
& = \iint K_f(x,y) \varphi(x) \psi(y) \dif x \dif y, \label{eq:CK-MLSchroEqu2018}
\end{align}
for all $\varphi$, $\psi \in \scrS(\Rn)$.
It is easy to show that $K_f(x,y) = \overline{K_f(y,x)}$.
Denote the symbol of $\mathfrak C_f$ as $c_f$. It can be verified the following identities hold in the distributional sense (cf. \cite{caro2016inverse}),
\begin{subequations} \label{eq:KandSymbol-MLSchroEqu2018}
	\begin{numcases}{}
	K_f(x,y) = (2\pi)^{-n} \int e^{i(x-y) \cdot \xi} c_f(x,\xi) \dif \xi, \label{eq:KtoSymbol-MLSchroEqu2018} \\
	c_f(x,\xi) = \int e^{-i\xi\cdot(x-y)} K_f(x,y) \dif y, \label{eq:SymboltoK-MLSchroEqu2018}
	\end{numcases}
\end{subequations}
where the integrals shall be understood as oscillatory integrals.
Despite the fact that $f$ usually is not a function, intuitively speaking, however, it is helpful to keep in mind the following correspondence,
\[
K_f(x,y) \sim \mathbb E_\omega \big( \overline{f(x,\omega)} f(y,\omega) \big).
\]
Moreover, we have
\[
K_f(x,y)
= \mu(x) |x-y|^{m-n} + F_\alpha(x,y),
\]
where $F_\alpha \in C^{0,\alpha}(\Rn \times \Rn)$ (cf. \cite{caro2016inverse}). Thus $\mu(x)$ plays a role similar to the covariance function of $f$.

We recall the domain $D_f$ in Theorem~\ref{thm:Unisigma-MLSchroEqu2018}. Through out the rest of the paper, for notational consistence, we let $\mathcal D$ be a bounded open  domain in $\mathbb{R}^3$ such that
\begin{equation} \label{eq:calD-MLSchroEqu2018}
\overline{D_f} \Subset \mathcal D.
\end{equation}
For a generalized Gaussian random field $f$, we define the so-called resolvent $\Rk f(x)$ as
\begin{equation} \label{eq:RkSigmaBDefn-MLSchroEqu2018}
\Rk f(x) := \agl[f, \Phi_k(x,\cdot)],
\end{equation}
where $\Phi_k(x,y) = \frac {e^{ik|x-y|}} {4\pi|x-y|}$ is the fundamental solution of the Helmholtz equation in $\R^3$, and we may abbreviate $\Phi_k$ as $\Phi$ if no ambiguity occurs.
We may also express $\Rk f(x)$ as an integral form $\int_{\R^3} \Phi_k(x,y) f(y) \dif{y}$.
The following lemma shows some preliminary properties of $\Rk f$. Note that the $\mu$ is the rough strength of $f$.

\begin{lem} \label{lemma:RkfBdd-MLSchroEqu2018}
	In dimension $n = 3$, we have $\Rk f \in L_{-1/2-\epsilon}^2$ for any $\epsilon > 0$ almost surely, and $\mathbb{E} ( \nrm[L^2(\mathcal D)]{\Rk f}) < C < +\infty$ for some constant $C$ independent of $k$.
\end{lem}

\begin{proof}
	We split $\Rk f$ into two parts, $\Rk ( \mathbb{E}f)$ and $\Rk (f - \mathbb{E}f)$.
	\cite[Lemma 2.1]{HLSM2019determining} gives $\Rk ( \mathbb{E}f) \in L_{-1/2-\epsilon}^2$.	
	
	As explained at the beginning of the article, the convolution $\Rk (f - \mathbb{E}f)(x,\omega)$ is measurable on the product space $\R^3 \times \Omega$, so does $\Rk (f - \mathbb{E}f)(x,\omega)$. Hence, we can exchange the order of integration with respect to $x$ and $\omega$ when we integrate.
	Therefore, by using \eqref{eq:CK-MLSchroEqu2018}, \eqref{eq:KandSymbol-MLSchroEqu2018} and \eqref{eq:RkSigmaBDefn-MLSchroEqu2018}, one can compute
	\begin{align}
	& \mathbb{E} ( \nrm[L_{-1/2-\epsilon}^2]{\Rk (f - \mathbb{E}f)(\cdot,\omega)}^2 ) \nonumber\\
	= & \int_{\R^3} \agl[x]^{-1-2\epsilon} \mathbb{E} ( \agl[\overline {f - \mathbb{E}f}, \Phi_{-k,x}] \agl[f - \mathbb{E}f, \Phi_{k,x}] ) \dif{x} = \int_{\R^3} \agl[x]^{-1-2\epsilon} \agl[\mathfrak C_f \Phi_{-k,x}, \Phi_{k,x}] \dif{x} \nonumber\\
	= & \int \agl[x]^{-1-2\epsilon} \int \big( (2\pi)^{-3} \int \int e^{i(y-z) \cdot \xi} c_f(y,\xi) \cdot \Phi_{-k,x}(z) \dif{z} \dif{\xi} \big) \Phi_{k,x}(y) \dif{y} \dif{x} \nonumber\\
	{\color{black}\simeq }& \int \agl[x]^{-1-2\epsilon} \int_{D_f} \big( \int_{D_f} \frac{\mathcal I(y,z) e^{-ik|x-z|}}{|x-z| \cdot |y-z|^2} \dif{z} \big) \cdot \frac{e^{ik|x-y|}}{|x-y|} \dif{y} \dif{x}, \label{eq:gkInter1-MLSchroEqu2018}
	\end{align}
	where $c_f(y,\xi)$ is the symbol of the covariance operator $\mathfrak C_f$ and
	\[
	\mathcal I(y,z) := \int_{\R^3} |y-z|^2 e^{i(y-z) \cdot \xi} c_f(y,\xi) \dif{\xi}.
	\]
	When $y = z$, we know $\mathcal I(y,z) = 0$ because the integrand is zero. Thanks to the condition $m > 2$, when $y \neq z$ we have
	\begin{align}
	|\mathcal I(y,z)| & = \big| \sum_{j=1}^3 \int_{\R^3} (y_j - z_j)^2 e^{i(y-z) \cdot \xi} c_f(y,\xi) \dif{\xi} \big| = \big| \sum_{j=1}^3 \int_{\R^3} e^{i(y-z) \cdot \xi} (\partial_{\xi_j}^2 c_f)(y,\xi) \dif{\xi} \big| \nonumber\\
	& \leq \sum_{j=1}^3 \int_{\R^3} C_j \agl[\xi]^{-m-2} \dif{\xi} \leq C_0 < +\infty, \label{eq:gkInter2-MLSchroEqu2018}
	\end{align}
	for some constant $C_0$ independent of $y$ and $z$.
	Note that $D_f \subset \R^3$ is bounded, so for $j=1,2$ we have
	\begin{equation} \label{eq:gkInter3-MLSchroEqu2018}
	\int_{D_f} |x-y|^{-j} \dif{y} \leq C_{f,j} \agl[x]^{-j}, \quad \forall x \in \R^3,
	\end{equation}
	for some constant $C_{f,j}$ depending only on $f, j$ and the dimension.
	The notation  $\agl[x]$ in \eqref{eq:gkInter3-MLSchroEqu2018} stands for $(1 + |x|^2)^{1/2}$ and readers may note the difference between the $\agl[\cdot]$ and the $\agl[\cdot,\cdot]$ appeared in \eqref{eq:RkSigmaBDefn-MLSchroEqu2018}.
	With the help of \eqref{eq:gkInter2-MLSchroEqu2018} and \eqref{eq:gkInter3-MLSchroEqu2018} and H\"older's inequality, we can continue \eqref{eq:gkInter1-MLSchroEqu2018} as
	\begin{align*}
	& \mathbb{E} ( \nrm[L_{-1/2-\epsilon}^2]{\Rk (f - \mathbb{E}f)(\cdot,\omega)}^2 ) \nonumber\\
	\lesssim & \int \agl[x]^{-1-2\epsilon} \big( \iint_{D_f \times D_f} (|x-z|^{-1} \cdot |y-z|^{-1}) (|y-z|^{-1} \cdot |x-y|)^{-1} \dif{z} \dif{y} \big) \dif{x} \nonumber\\
	\leq & \int \agl[x]^{-1-2\epsilon} \big[ C \int_{D_f} ( \int_{D_f} |y-z|^{-2} \dif{y} ) |x-z|^{-2} \dif{z} \nonumber\\
	& \hspace{13ex} \cdot \int_{D_f} ( \int_{D_f} |y-z|^{-2} \dif{z} ) |x-y|^{-2} \dif{y} \big]^{1/2} \dif{x} \nonumber\\
	= & \int \agl[x]^{-1-2\epsilon} \big( C_f \int_{D_f} |x-z|^{-2} \dif{z} \cdot \int_{D_f} |x-y|^{-2} \dif{y} \big)^{1/2} \dif{x} \qquad (\text{by~} \eqref{eq:gkInter3-MLSchroEqu2018}) \nonumber\\
	= & \int \agl[x]^{-1-2\epsilon} C_f \agl[x]^{-2} \dif{x}  \leq C_f < +\infty,
	\end{align*}
	which gives
	\begin{equation} \label{eq:RkfBounded-MLSchroEqu2018}
	\mathbb{E} ( \nrm[L_{-1/2-\epsilon}^2]{\Rk (f - \mathbb{E}(f))(\cdot,\omega)}^2 ) \leq C_f < +\infty.
	\end{equation}
	By the H\"older inequality applied to the probability measure, we obtain from \eqref{eq:RkfBounded-MLSchroEqu2018} that
	\begin{equation} \label{eq:RkfBoundedCD-MLSchroEqu2018}
	\mathbb{E} \nrm[L_{-1/2-\epsilon}^2]{\Rk (f - \mathbb{E}(f))} \leq [ \mathbb{E} ( \nrm[L_{-1/2-\epsilon}^2]{\Rk (f - \mathbb{E}(f))}^2 ) ]^{1/2} \leq C_f^{1/2} < +\infty,
	\end{equation}
	for some constant $C_f$ independent of $k$. The formula \eqref{eq:RkfBoundedCD-MLSchroEqu2018} gives that
	\(
	\Rk (f - \mathbb{E}(f)) \in L_{-1/2-\epsilon}^2
	\)
	almost surely, and hence
	\[
	\Rk f \in L_{-1/2-\epsilon}^2 \quad \ass.
	\]

	By replacing $\R^3$ with $\mathcal D$ and deleting the term $\agl[x]^{-1-2\epsilon}$ in the derivation above, one easily arrives at $\mathbb{E} \nrm[L^2(\mathcal D)]{\Rk f} < +\infty$.
	The proof is complete. 
\end{proof}

The following resolvent estimate
\begin{equation} \label{eq:kVb-MLSchroEqu2018}
\nrm[L_{-1/2-\epsilon}^2(\R^3)]{\Rk \varphi} \lesssim k^{-1} \nrm[L_{1/2+\epsilon}^2(\R^3)]{\varphi}
\end{equation}
is known to the literature (cf. \cite{eskin2011lectures, hormander1985analysisIII}). In what follows, we shall also need some variations of it for our arguments.

\begin{lem} \label{lem:RkVb-MLSchroEqu2018}
	There exists a constant $k_0 > 0$ depending on $\epsilon$ and $V$ such that for $\forall k > k_0$ and multi-index $\alpha \colon |\alpha| \leq 2$, we have
	\begin{align}
	\nrm[L_{-1/2-\epsilon}^2(\R^3)]{\Rk (\partial^\alpha V) \varphi} & \leq C k^{-1} \nrm[L_{-1/2-\epsilon}^2(\R^3)]{\varphi},  \label{eq:RVb1-MLSchroEqu2018} \\
	\nrm[L_{1/2+\epsilon}^2(\R^3)]{(\partial^\alpha V) \Rk \varphi} & \leq C k^{-1} \nrm[L_{1/2+\epsilon}^2(\R^3)]{\varphi}, \label{eq:RVb2-MLSchroEqu2018} \\
	\nrm[L^1(\R^3)]{(\partial^\alpha V) \Rk \varphi} & \leq C k^{-1} \nrm[L_{1/2+\epsilon}^2(\R^3)]{\varphi}.  \label{eq:RVb3-MLSchroEqu2018}
	\end{align}
\end{lem}

\begin{proof}
	Recall the assumption on $V$.
	We only show the case where $\alpha = 0$.
	With the help of \eqref{eq:kVb-MLSchroEqu2018}, we can have
	\begin{equation*}
	\nrm[L_{-1/2-\epsilon}^2(\R^3)]{\Rk V \varphi}
	\lesssim k^{-1} \nrm[L_{1+2\epsilon}^\infty(\R^3)]{V} \nrm[L_{-1/2-\epsilon}^2(\R^3)]{\varphi}
	\lesssim k^{-1} \nrm[L_{-1/2-\epsilon}^2(\R^3)]{\varphi},
	\end{equation*}
	and
	\begin{equation*}
	\nrm[L_{1/2+\epsilon}^2(\R^3)]{V \Rk \varphi}
	\leq \nrm[L^\infty(\R^3)]{\agl[x]^{1+2\epsilon} V} \nrm[L_{-1/2-\epsilon}^2(\R^3)]{\Rk \varphi}
	\lesssim k^{-1} \nrm[L_{1/2+\epsilon}^2(\R^3)]{\varphi}.
	\end{equation*}
	Moreover, by H\"older's inequality we can have
	\begin{equation*}
	\nrm[L^1(\R^3)]{V \Rk \varphi}
	\leq \nrm[L^2]{\agl[x]^{1/2+\epsilon} V} \cdot \nrm[L_{-1/2-\epsilon}^2]{\Rk \varphi}
	\lesssim k^{-1} \nrm[L_{1/2+\epsilon}^2(\R^3)]{\varphi}.
	\end{equation*}
	The proof is complete.
\end{proof}

\subsection{The well-posedness of the direct problem} \label{subset:WellDefined-MLSchroEqu2018}

For a particular realization of the random sample $\omega \in \Omega$, the regularity of an m.i.g.r.~$f$ could be very rough; 
see Lemma \ref{lem:migrRegu-MLSchroEqu2018}.
Due to this reason, the classical second-order elliptic PDE theory may no longer be applicable to \eqref{eq:1-MLSchroEqu2018}, and next we shall introduce $u$ as a distributional solution.
Reformulating \eqref{eq:1-MLSchroEqu2018} into the Lippmann-Schwinger equation formally (cf.~\cite{colton2012inverse}), we have
\begin{equation} \label{eq:uscDefn-MLSchroEqu2018}
(I - \Rk V) u = - \Rk f,
\end{equation}
where the term $\Rk f$ is defined by \eqref{eq:RkSigmaBDefn-MLSchroEqu2018}. 
Suppose that $k$ is large enough, then we know the operator $I - \Rk V$ is an invertible mapping from $L_{-1/2-\epsilon}^2$ to $L_{-1/2-\epsilon}^2$.
Moreover, by Lemma \ref{lemma:RkfBdd-MLSchroEqu2018} we know that the right-hand side of \eqref{eq:uscDefn-MLSchroEqu2018} belongs to $L_{-1/2-\epsilon}^2$ almost surely.
We are now in a position to present one of the results concerning the direct scattering problem.

\begin{prop} \label{prop:MildSolUnique-MLSchroEqu2018}
	When $k$ is large enough such that $\nrm[\mathcal{L}(L_{-1/2-\epsilon}^2, L_{-1/2-\epsilon}^2)]{\Rk V} < 1$,
	for almost every $\omega \in \Omega$, there exists a unique distributional solution such that $u(x)$ satisfies \eqref{eq:uscDefn-MLSchroEqu2018} a.s..
	Moreover, ${u(\cdot,\omega) \in L_{-1/2-\epsilon}^2}$ a.s.~for any $\epsilon\in\mathbb{R}_+$, and there holds
	\[
	\nrm[L_{-1/2-\epsilon}^2]{u}
	\leq C \nrm[L_{-1/2-\epsilon}^2]{\Rk f}.
	\]
	for a certain constant $C$ depending on $V$.
\end{prop}

\begin{proof}
	By Lemma \ref{lemma:RkfBdd-MLSchroEqu2018}, we obtain
	$$F := -\Rk f \in L_{-1-\epsilon}^2.$$
	According to \eqref{eq:RVb1-MLSchroEqu2018} we have $\nrm[\mathcal{L}{(L_{-1/2-\epsilon}^2, L_{-1/2-\epsilon}^2)}]{\Rk V} < 1$. 
	Hence, $\sum_{j=0}^\infty (\Rk V)^j$ is well-defined. 
	Therefore, $\sum_{j=0}^\infty (\Rk V)^j F \in L_{-1/2-\epsilon}^2$. 
	Because $\sum_{j=0}^\infty (\Rk V)^j = (I - \Rk V)^{-1}$, we see $(I - \Rk V)^{-1} F \in L_{-1/2-\epsilon}^2$. 
	Let $u := (I - \Rk V)^{-1} F \in L_{-1/2-\epsilon}^2$, then $u$ fulfils the requirements. 
	Hence, the existence of a solution is proved.
	The uniqueness and stability follows easily from the inequality
	\[
	\nrm[L_{-1/2-\epsilon}^2]{u} \leq \sum_{j \geq 0} \nrm[\mathcal L(L_{-1/2-\epsilon}^2,L_{-1/2-\epsilon}^2)]{\Rk V}^j \nrm[L_{-1/2-\epsilon}^2]{\Rk f} \leq C \nrm[L_{-1/2-\epsilon}^2]{\Rk f}.
	\]
	The proof is complete. 
\end{proof}

Next we show that the far-field pattern is well-defined in the $L^2$ sense. Assume that $k$ is large enough. From \eqref{eq:uscDefn-MLSchroEqu2018} we deduce that
\[
u = -(I - \Rk V)^{-1} (\Rk f) = -\Rk (I - V \Rk)^{-1} (f).
\]
Therefore, we define the far-field pattern of the scattered wave $u(x,k,\omega)$ formally in the following manner,
\begin{equation} \label{eq:uInftyDefn-MLSchroEqu2018}
u^\infty(\hat x,k,\omega) := \frac {-1} {4\pi} \int_{\R^3} e^{-ik\hat x \cdot y} (I - V \Rk)^{-1} (f)(y) \dif{y}, \quad \hat x \in \mathbb{S}^2.
\end{equation}

\begin{prop} \label{prop:FarFieldWellDefined-MLSchroEqu2018}
	Define the far-field pattern of the solution as in \eqref{eq:uInftyDefn-MLSchroEqu2018}. When $k$ is large enough, there is a subset $\Omega_0 \subset \Omega$, 
	with zero measure $\mathbb P (\Omega_0) = 0$, such that it holds
	$$u^\infty(\hat x, k, \omega) \in L^2(\mathbb{S}^2),\ \  \Forall \omega \in \Omega \backslash \Omega_0.$$
\end{prop}

\begin{proof}
	By \cite[Lemma 2.4]{HLSM2019determining},  we have
	\[
	\nrm[\mathcal{L}(L^2(\mathcal D), L^2(\mathcal D))]{V \Rk} \leq C k^{-1} < 1,
	\]
	when $k$ is sufficiently large. Therefore, it holds that
	\begin{align}\label{eq:a1-MLSchroEqu2018}
	& \ \int_{\mathbb S^2} |u^\infty(\hat x, k, \omega)|^2 \dif{S(\hat x)} \nonumber\\
	\lesssim & \ \int_{\mathbb S^2} \big| \int_{\R^3} e^{-ik\hat x \cdot y} (I - V \Rk)^{-1} (f) \dif{y} \big|^2 \dif{S(\hat x)} \nonumber\\
	\lesssim & \ \int \big| \int_{\R^3} e^{-ik\hat x \cdot y} \sum_{j \geq 1} (V \Rk)^j (f) \dif{y} \big|^2 \dif{S(\hat x)} + \int |\agl[f,e^{-ik\hat x \cdot (\cdot)}]|^2 \dif{S(\hat x)} \nonumber\\
	=: & \ f_1(\hat x, k,\omega) + f_2(\hat x, k,\omega).
	\end{align}
	
	Next, we derive estimates on these terms $f_j ~(j=1,2)$ in \eqref{eq:a1-MLSchroEqu2018}.
	We have
	\begin{align*}
	& |\int_{\R^3} e^{-ik\hat x \cdot y} \sum_{j \geq 1} (V \Rk)^j (f) \dif y|
	\leq \sum_{j \geq 0} \nrm[L_{1/2+\epsilon}^2(\R^3)]{V} \nrm[L_{-1/2-\epsilon}^2(\R^3)]{(\Rk V)^j \Rk f} \\
	\lesssim & \ \sum_{j \geq 0} k^{-j} \nrm[L_{-1/2-\epsilon}^2(\R^3)]{\Rk f}
	\lesssim \nrm[L_{-1/2-\epsilon}^2(\R^3)]{\Rk f},
	\end{align*}
	where we have used the assumption \eqref{eq:asp-MLSchroEqu2018}, eq.~\eqref{eq:RVb1-MLSchroEqu2018} and Lemma \ref{lemma:RkfBdd-MLSchroEqu2018}.
	Therefore,
	\begin{equation*}
	f_1(\hat x, k, \omega)
	\lesssim \int  \nrm[L_{-1/2-\epsilon}^2(\R^3)]{\Rk f}^2 \dif{S(\hat x)}
	< C < +\infty,
	\end{equation*}
	for some constant $C$ independent of $k$.
	
	By \eqref{eq:CK-MLSchroEqu2018} and Fubini's theorem, the expectation of $f_2(\hat x, k, \omega)$ can be computed as
	\begin{align*}
	\mathbb{E} f_2(\hat x, k, \omega)
	& = \mathbb E \int |\agl[f,e^{-ik\hat x \cdot (\cdot)}]|^2 \dif{S(\hat x)} 
	= \int \mathbb E |\agl[f,e^{-ik\hat x \cdot (\cdot)}]|^2 \dif{S(\hat x)} \\
	& = \int |\agl[\mathfrak C_f (\chi_{D_f} e^{-ik\hat x \cdot (\cdot)}), (\chi_{D_f} e^{ik\hat x \cdot (\cdot)})]| \dif{S(\hat x)} \\
	& \quad + \int_{\mathbb{S}^2} \iint_{\mathcal D \times \mathcal D} \mathbb{E}f(y) \overline{\mathbb{E}f(z)} e^{-ik \hat x \cdot (y-z)} \dif y \dif z \dif{S(\hat x)} \\
	& \leq \int \nrm[L^2(\R^3)]{\mathfrak C_f (\chi_{D_f} e^{-ik\hat x \cdot (\cdot)})} \cdot \nrm[L^2(\R^3)]{\chi_{D_f} e^{ik\hat x \cdot (\cdot)}} \dif{S(\hat x)} + C_f,
	\end{align*}
	where the constant $C_f$ is independent of $\hat x$ and $k$.
	The symbol of the pseudo-differential operator is of order $-m < 0$, thus $\mathfrak C_f$ is a bounded operator from $L^2(\R^3)$ to $L^2(\R^3)$; see \cite[Theorem 11.7]{wong2014pdo}. Hence
	\begin{align*}
	\mathbb{E} f_2(\hat x, k, \omega)
	& \leq C \int \nrm[L^2(\R^3)]{\chi_{D_f} e^{-ik\hat x \cdot (\cdot)}} \cdot \nrm[L^2({D_f})]{\chi_{D_f} e^{ik\hat x \cdot (\cdot)}} \dif{S(\hat x)} + C_f \\
	& \leq C \int \nrm[L^2(\R^3)]{\chi_{D_f}} \cdot \nrm[L^2({D_f})]{\chi_{D_f}} \dif{S(\hat x)} + C_f \\
	& \leq C_f < +\infty,
	\end{align*}
	for some constant $C_f$ independent of $\hat x$ and $k$. 
	Thus, $f_2(\hat x, k, \omega) < +\infty$ almost surely.
	
	Combining the estimates on $f_j(\hat x, \omega) ~(j=1,2)$, we conclude that
	\[
	\int_{\mathbb S^2} |u^\infty(\hat x,k,\omega)|^2 \dif{S(\hat x)}< \infty
	\]
	almost surely.
	The proof is complete. 
\end{proof}


\section{Several critical asymptotic estimates} \label{sec:AsyEst-MLSchroEqu2018}

In this section we shall establish a method to recover rough strength of $f$ through the following quantity
\begin{equation}\label{eq:rcv1}
\frac 1 {K} \int_{K}^{2K} \overline{u^\infty(\hat x,k,\omega)} \cdot u^\infty(\hat x,k+\tau,\omega) \dif{k}.
\end{equation}
As an auxiliary critical step in justifying \eqref{eq:rcv1}, we need to first consider the following recovery formula 	
\begin{equation} \label{eq:example2-MLSchroEqu2018}
\frac 1 {K} \int_{K}^{2K} \overline{[u^\infty(\hat x,k,\omega) - \mathbb{E}(u^\infty(\hat x,k))]} \cdot [u^\infty(\hat x,k,\omega) - \mathbb{E}(u^\infty(\hat x,k))] \dif{k}.
\end{equation}	
It is noted that $\mathbb{E}(u^\infty(\hat x,k))$ requires the full realization of the random samples. We would like to emphasise that $\mathbb{E}(u^\infty(\hat x,k))$ shall play an auxiliary role in our analysis and we shall develop techniques to remove it from the recovery procedure.

To analyze the behaviour of \eqref{eq:example2-MLSchroEqu2018}, we shall derive several critical asymptotic estimates in this section. 
Henceforth, we use $k^*$ to signify the maximum value between the quantity $k_0$ from Lemma \ref{lem:RkVb-MLSchroEqu2018} and the quantity
\[
\sup_{k \in \R_+} \{ k \,;\, \nrm[\mathcal L(L_{-1/2 - \epsilon}^2, L_{-1/2 - \epsilon}^2)]{\mathcal{R}_k V} \geq 1 \} + 1.
\]
Assume that $k > k^*$, then we can expand $(I - V\Rk)^{-1}$ into Neumann series and obtain
\begin{align}
u^\infty(\hat x,k,\omega) - \mathbb{E}(u^\infty(\hat x,k))
\ = & \ \frac {-1} {4\pi} \sum_{j=0}^{+\infty} \int_{\R^3} e^{-ik\hat x \cdot y} (V \Rk)^j (f - \mathbb{E}(f))(y) \dif{y}, \quad \hat x \in \mathbb{S}^2 \nonumber\\
:= & \ \frac {-1} {4\pi} \big[ F_0(k,\hat x) + F_1(k,\hat x) \big], \label{eq:u1InftyDefn-MLSchroEqu2018}
\end{align}
where
\begin{equation} \label{eq:Fjkx-MLSchroEqu2018}
\left\{\begin{aligned}
F_0(k,\hat x,\omega) & := \agl[f - \mathbb{E}(f), e^{-ik\hat x \cdot (\cdot)}], \\
F_1(k,\hat x,\omega) & := \sum_{j \geq 1} \int_{\R^3} e^{-ik \hat x \cdot y} (V \Rk)^j (f - \mathbb{E}(f))(y) \dif{y}.
\end{aligned}\right.
\end{equation}
The expectation $\mathbb{E}(u^\infty(\hat x,k))$ in \eqref{eq:u1InftyDefn-MLSchroEqu2018} can be expressed as
\begin{equation} \label{eq:u2InftyDefn-MLSchroEqu2018}
\mathbb E (u^\infty(\hat x,k)) = \frac {-1} {4\pi} \int_{\R^3} e^{-ik\hat x \cdot y} \big( (I - V \Rk)^{-1} \mathbb E(f) \big)(y) \dif{y}, \quad \hat x \in \mathbb{S}^2.
\end{equation}

\begin{lem} \label{lemma:FarFieldGoToZero-MLSchroEqu2018}
	For $\forall k > k^*$, there exists a constant $C$ independent of $\hat x$ and $k$ such that
	\begin{equation}\label{eq:estn1}
	|\mathbb E (u^\infty(\hat x,k))| \leq C k^{-2}.
	\end{equation}
\end{lem}
\begin{proof}
	Note that $\mathbb{E} f \in \mathcal{C}_c^\infty(\R^3)$ (cf.~Definition \ref{defn:migr-MLSchroEqu2018}). 
	The function $\Rk (\mathbb E f)$ is a convolution and thus is a $\mathcal{C}^\infty$-smooth function.
	For $k > k^*$ we denote $\mathbf{F}(y) := \sum_{j = 0}^{1} (V \Rk)^j (\mathbb E f) (y)$ for simplicity.
	By using \eqref{eq:u2InftyDefn-MLSchroEqu2018} and Lemma \ref{lem:RkVb-MLSchroEqu2018}, we can compute
	\begin{align*}
	|\mathbb E (u^\infty(\hat x,k))|
	& \leq |\int [(ik^{-1} \hat x \cdot \nabla_y)^2 e^{-ik\hat x \cdot y}] \mathbf{F}(y) \dif{y}| + \sum_{j \geq 2} |\int e^{-ik\hat x \cdot y} (V \Rk)^j (\mathbb E f) (y) \dif{y}| \\
	& \leq C k^{-2} \int |\sum_{|\alpha| = 2} C_\alpha \partial_y^\alpha \mathbf{F}(y)| \dif{y} + \nrm[L_{1/2+\epsilon}^2]{V} \sum_{j \geq 1} \nrm[L_{-1/2-\epsilon}^2]{(\Rk V)^j \Rk (\mathbb E f)} \\
	& \leq C k^{-2} \int |\sum_{|\alpha| = 2} C_\alpha \partial_y^\alpha \mathbf{F}(y)| \dif{y} + C k^{-2} \cdot \nrm[L_{1/2+\epsilon}^2(\R^3)]{\mathbb E f},
	\end{align*}
	where the constant $C$ is independent of $\hat x$ and $k$.
	Note that $\mathbb E f \in C_c^\infty$ so $\partial_j (\mathbb E f) \in C_c^\infty$ and $\partial_j \Rk (\mathbb E f) = \Rk \partial_j (\mathbb E f)$, thus	by Lemma \ref{lem:RkVb-MLSchroEqu2018} {and the assumption \eqref{eq:asp-MLSchroEqu2018}} we have	
	\begin{equation*}
	\int |\sum_{|\alpha| = 2} C_\alpha \partial_y^\alpha \mathbf{F}(y)| \dif{y}
	\lesssim \sum_{|\alpha| = 2} \nrm[L^1(\R^3)]{\partial^\beta \mathbb E f} + \sum_{|\alpha| + |\beta| = 2} \nrm[L^1(\R^3)]{(\partial^\alpha V) \Rk (\partial^\beta \mathbb E f)}
	< +\infty.
	\end{equation*}
	The proof is complete. 
\end{proof}

By substituting \eqref{eq:u1InftyDefn-MLSchroEqu2018}, \eqref{eq:Fjkx-MLSchroEqu2018} into \eqref{eq:example2-MLSchroEqu2018}, we obtain several crossover terms between $F_0$ and $F_1$. 
The asymptotic estimates of these crossover terms are the main purpose of Sections \ref{subsec:AELeading-MLSchroEqu2018} and \ref{subsec:AsympHighOrder-MLSchroEqu2018}.
Section \ref{subsec:AELeading-MLSchroEqu2018} focuses on the estimate of the leading-order term while the estimates of the higher-order terms are presented in Section \ref{subsec:AsympHighOrder-MLSchroEqu2018}.

\subsection{Asymptotics of the leading-order term} \label{subsec:AELeading-MLSchroEqu2018}

The following lemma is needed for the study of the asymptotics of the aforementioned leading-order term.

\begin{lem} \label{lemma:LeadingTermTechnical-MLSchroEqu2018}
	When $\mu \in \mathcal{C}_c^\infty(\mathcal D)$, $\tau \in \R$ and $\hat x \in \mathbb{S}^2$, we have
	\begin{align}
	\frac {(2\pi)^3} {K^2} \int_K^{2K} \int_K^{2K} |\widehat{\mu}((k_1 - k_2) \hat x)|^2 \dif{k_1} \dif{k_2} & \leq CK^{-1}, \label{eq:F0F0TermOne-MLSchroEqu2018} \\
	\frac {(2\pi)^3} {K^2} \int_K^{2K} \int_K^{2K} |\widehat{\mu}((k_1 + k_2 + \tau) \hat x)|^2 \dif{k_1} \dif{k_2} & \leq CK^{-1}, \label{eq:F0F0TermTwo-MLSchroEqu2018}
	\end{align}
	for some constant $C$ independent of $\tau$ and $\hat x$.
	Here $\widehat \mu$ signifies the Fourier transform of $\mu$.
\end{lem}

\begin{proof}
	To conclude \eqref{eq:F0F0TermOne-MLSchroEqu2018}, we make a change of variable,
	\begin{equation*}
	\left\{\begin{aligned}
	s & = k_1 - k_2, \\
	t & = k_2.
	\end{aligned}\right.
	\end{equation*}
	Write $Q = \{(s,t) \in \R^2 \,\big|\, K \leq s+t \leq 2K,\, K \leq t \leq 2K \}$, which is illustrated in Fig.~\ref{fig:D-MLSchroEqu2018}.
	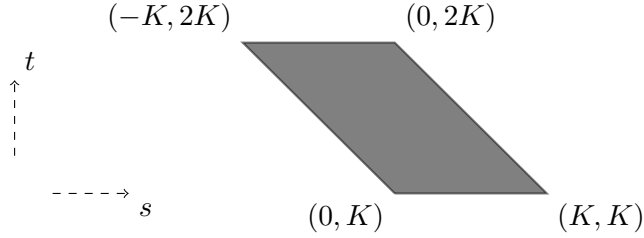
\begin{figure}[h]
		\centering
		\begin{tikzpicture}
		\draw[dashed,->] (-4.5,-1) -- (-3.5,-1) node[anchor=north west] {$s$};
		\draw[dashed,->] (-5,-0.5) -- (-5,0.5) node[anchor=south west] {$t$};
		\filldraw[opacity=0.5, thick] (0,-1) -- (2,-1) -- (0,1) -- (-2,1) -- (0,-1);
		\filldraw[black] (0,-1) node[anchor=north east] {$(0,K)$};
		\filldraw[black] (2,-1) node[anchor=north west] {$(K,K)$};
		\filldraw[black] (0, 1) node[anchor=south west] {$(0,2K)$};
		\filldraw[black] (-2,1) node[anchor=south east] {$(-K,2K)$};
		\end{tikzpicture}
		\caption{Schematic illustration of $Q$. } \label{fig:D-MLSchroEqu2018}
	\end{figure}
	Recall that $\supp \mu \subseteq D_f$. Then we have
	\begin{align}
	& \frac 1 {K^2} \int_K^{2K} \int_K^{2K} |\widehat{\mu}((k_1 - k_2) \hat x)|^2 \dif{k_1} \dif{k_2} = \frac 1 {K^2} \iint_Q |\widehat{\mu}(s \hat x)|^2 \dif{s} \dif{t} \nonumber\\
	= & \frac 1 {K^2} \int_{-K}^0 (K+s) |\widehat{\mu}(s \hat x)|^2 \dif{s} + \frac 1 {K^2} \int_0^{K} (K-s) |\widehat{\mu}(s \hat x)|^2 \dif{s} \nonumber\\
	\leq & \frac 2 K \int_\R |\widehat{\mu}(s \hat x)|^2 \dif{s}. \label{eq:muInter1-MLSchroEqu2018}
	\end{align}
	Recall that $\mu \in \mathcal{C}_c^\infty(\R^3) \subset \scrS(\R^3)$, thus $\widehat{\mu}(x)$ decays faster than the reciprocal of any polynomials, especially, $|\widehat{\mu}(s \hat x)| \leq C \agl[s]^{-1}$ for all $\forall s \in \R$, thus
	$$\frac 1 {K^2} \int_K^{2K} \int_K^{2K} |\widehat{\mu}((k_1 - k_2) \hat x)|^2 \dif{k_1} \dif{k_2} \leq \frac 2 K \int_\R C \agl[s]^{-2} \dif{s} \leq  C K^{-1},$$
	which is \eqref{eq:F0F0TermOne-MLSchroEqu2018}.
	To prove \eqref{eq:F0F0TermTwo-MLSchroEqu2018}, again we make a change of variable:
	\begin{equation*} 
	\left\{\begin{aligned}
	s & = k_1 + k_2 + \tau, \\
	t & = k_2.
	\end{aligned}\right.
	\end{equation*}
	Write $Q' = \{(s,t) \in \R^2 \,\big|\, K \leq s-t-\tau \leq 2K,\, K \leq t \leq 2K \}$.
	One can compute
	\begin{align*}
	& \frac 1 {K^2} \int_K^{2K} \int_K^{2K} |\widehat{\mu}((k_1 + k_2 + \tau) \hat x)|^2 \dif{k_1} \dif{k_2} = \frac 1 {K^2} \iint_{Q'} |\widehat{\mu}(s \hat x)|^2 \dif{s} \dif{t} \\
	= & \frac 1 {K^2} \int_{2K+\tau}^{3K+\tau} (s-2K-\tau) |\widehat{\mu}(s \hat x)|^2 \dif{s} + \frac 1 {K^2} \int_{3K+\tau}^{4K+\tau} (4K+\tau-s) |\widehat{\mu}(s \hat x)|^2 \dif{s} \\
	\leq & \frac 2 {K} \int_{2K-\tau}^{2K+\tau} |\widehat{\mu}(s \hat x)|^2 \dif{s} = \frac 2 {K} \int_\R |\widehat{\mu}(s \hat x)|^2 \dif{s} \leq \frac C K \int_\R \agl[s]^{-2} \dif{s} \leq \frac C K,
	\end{align*}
	which gives \eqref{eq:F0F0TermTwo-MLSchroEqu2018}.
	The proof is complete.
\end{proof}

For notational convenience, we shall use $\{K_j\} \in P(t)$ to signify a sequence $\{K_j\}_{j \in \mathbb{N}}$ satisfying $K_j \geq C j^t ~(j \in \mathbb{N})$ for some fixed constant $C > 0$. Throughout the rest of the paper, $\gamma$ stands for a fixed positive real number.	
The next lemma gives the asymptotic estimate of the leading-order term.

\begin{lem} \label{lemma:LeadingTermErgo-MLSchroEqu2018}
	Let $F_j(k,\hat x) ~(j=0,1)$ be defined as in \eqref{eq:Fjkx-MLSchroEqu2018}. Write
	\begin{equation*}
	X_{0,0}(K,\tau,\hat x) = \frac 1 K \int_K^{2K} k^m \overline{F_0(k,\hat x)} \cdot F_0(k+\tau,\hat x) \dif{k}.
	\end{equation*}
	Assume that $\{K_j\} \in P(1+\gamma)$. Then for any $\tau > 0$, we have
	\begin{equation} \label{eq:LeadingTermErgo-MLSchroEqu2018}
	\lim_{j \to +\infty} X_{0,0}(K_j,\tau,\hat x)
	= (2\pi)^{3/2} \widehat{\mu} (\tau \hat x) \quad \ass.
	\end{equation}
\end{lem}

The proof of Lemma \ref{lemma:LeadingTermErgo-MLSchroEqu2018} utilizes ergodicity.
In what follows, we may denote $X_{0,0}(K,\tau,\hat x)$ as $X_{0,0}$ for short if it is clear in the context.

\begin{proof}[Proof of Lemma \ref{lemma:LeadingTermErgo-MLSchroEqu2018}]
	
	By \eqref{eq:CK-MLSchroEqu2018}, \eqref{eq:KandSymbol-MLSchroEqu2018} and \eqref{eq:Fjkx-MLSchroEqu2018}, we can compute $\mathbb{E} \big( \overline{F_0(k,\hat x)} F_0(k+\tau,\hat x) \big)$ as follows,
	\begin{align}
	& \ \mathbb{E} \big( \overline{F_0(k,\hat x)} F_0(k+\tau,\hat x) \big) \nonumber \\
	= & \ \int_{\R^3} \Big( \int_{\R^3} K_f(y,z) e^{-ik \hat x \cdot (y-z)} \dif{z} \Big) e^{-i \tau \hat x \cdot y} \dif{y} \nonumber \\
	= & \ \int_{\R^3} c_f(y,k \hat x) e^{-i\tau \hat x \cdot y} \dif{y}
	= (2\pi)^{3/2}\, \widehat{\mu} (\tau \hat x) k^{-m} + \int_{\R^3} a(y,k \hat x) e^{i\tau \hat x \cdot y} \dif{y}, \label{eq:I0-MLSchroEqu2018}
	\end{align}
	where $a(y,\xi) := c_f(y,\xi) - \mu(y)|\xi|^{-m}$ is the difference of the full symbol and the principal symbol.
	$a(y,k \hat x)$ is compactly supported in $y$ and $|a(y,k \hat x)| \lesssim k^{-m-1}$.
	Therefore,
	\begin{align}
	& \ \mathbb{E} ( X_{0,0} )
	= \frac 1 K \int_K^{2K} k^m \mathbb{E} \big( \overline{F_0(k,\hat x)} F_0(k+\tau,\hat x) \big) \dif{k} \nonumber \\
	= & \ \frac 1 K \int_K^{2K} [(2\pi)^{3/2} \,\widehat{\mu}(\tau \hat x) + \mathcal{O}(k^{-1})] \dif{k} \nonumber \\
	= & \ (2\pi)^{3/2} \,\widehat{\mu}(\tau \hat x) + \mathcal{O}(K^{-1}), \quad K \to +\infty. \label{eq:I0EstE-MLSchroEqu2018}
	\end{align}
	
	By Isserlis' Theorem and \eqref{eq:I0EstE-MLSchroEqu2018}, and noting that $\overline{F_j(k,\hat x)} = F_j(-k,\hat x)$, $F_0(-k,-\hat x) = F_0(k,\hat x)$, one can compute
	\begin{align}
	& \mathbb{E} \big( | X_{0,0} - (2\pi)^{3/2} \widehat{\mu} (\tau \hat x) |^2 \big) \nonumber\\
	= & \frac 1 {K^2} \int_K^{2K} \int_K^{2K} \mathbb{E} \Big( \big[ k_1^m F_0(k_1+\tau,\hat x) \overline{F_0(k_1,\hat x)} - (2\pi)^{3/2} \widehat{\mu}(\tau \hat x) \big] \nonumber\\
	& \times \big[ k_2^m \overline{ F_0(k_2+\tau,\hat x) } F_0(k_2,\hat x) - (2\pi)^{3/2} \overline{ \widehat{\mu}(\tau \hat x) } \big] \Big) \dif{k_1} \dif{k_2} \nonumber\\
	\leq & \frac {(2\pi)^3} {K^2} \int\limits_K^{2K} \int\limits_K^{2K} |\widehat{\mu}((k_1 - k_2) \hat x)|^2 \dif{k_1} \dif{k_2} + \frac {(2\pi)^3} {K^2} \int\limits_K^{2K} \int\limits_K^{2K} |\widehat{\mu}((k_1 + k_2 + \tau) \hat x)|^2 \dif{k_1} \dif{k_2} \nonumber\\
	& + \big( \frac 1 {K^2} \int_K^{2K} \int_K^{2K} |\widehat{\mu}((k_1 - k_2) \hat x)|^2 \dif{k_1} \dif{k_2} \big)^{1/2} \cdot \mathcal{O}(K^{-1}) + \mathcal{O}(K^{-1}). \label{eq:X00Square-MLSchroEqu2018}
	\end{align}
	Note that the missing term involving $\widehat{\mu}((k_1 + k_2 + \tau) \hat x)$ in \eqref{eq:X00Square-MLSchroEqu2018} is counted into $\mathcal{O}(K^{-1})$ because $\widehat{\mu}((k_1 + k_2 + \tau) \hat x) \to 0 ~(k_1,k_2 \to +\infty)$. 
	By \eqref{eq:X00Square-MLSchroEqu2018} and Lemma \ref{lemma:LeadingTermTechnical-MLSchroEqu2018}, we have
	\begin{equation} \label{eq:X00Bdd-MLSchroEqu2018}
	\mathbb{E} \big( | X_{0,0} - (2\pi)^{3/2} \widehat{\mu} (\tau \hat x) |^2 \big) = \mathcal{O}(K^{-1}), \quad K \to +\infty.
	\end{equation}
	Fixing an integer $K_0 > 0$ and by Chebyshev's inequality and \eqref{eq:X00Bdd-MLSchroEqu2018} we have
	\begin{align}
	& \ P \big( \bigcup_{j \geq K_0} \{ | X_{0,0}(K_j) - (2\pi)^{3/2} \widehat \mu (\tau \hat x) | \geq \epsilon \} \big) \nonumber \\
	\leq & \ \frac 1 {\epsilon^2} \sum_{j \geq K_0} \mathbb{E} \big( | X_{0,0}(K_j) - (2\pi)^{3/2} \widehat \mu (\tau \hat x) |^2 \big) \nonumber\\
	\lesssim & \ \frac 1 {\epsilon^2} \sum_{j \geq K_0} K_j^{-1} = \frac 1 {\epsilon^2} \sum_{j \geq K_0} j^{-1-\gamma} \leq \frac 1 {\epsilon^2} \int_{K_0}^{+\infty} (t-1)^{-1-\gamma} \dif{t} = \frac 1 {\epsilon^2 \gamma} (K_0-1)^{-\gamma}. \label{eq:PX00Epsilon-MLSchroEqu2018}
	\end{align}
	Here $X_{0,0}(K_j)$ stands for $X_{0,0}(K_j, \tau, \hat x)$. By \cite[Lemma 3.2]{HLSM2019determining}, formula \eqref{eq:PX00Epsilon-MLSchroEqu2018} implies that for any fixed $\tau \geq 0$ and $\hat x \in \mathbb{S}^2$, one has
	\[
	X_{0,0}(K_j,\tau,\hat x) \to (2\pi)^{3/2} \widehat \mu (\tau \hat x) \quad \ass.
	\]
	The proof is complete.
\end{proof}

\subsection{Asymptotics of the higher-order terms} \label{subsec:AsympHighOrder-MLSchroEqu2018}

\begin{lem} \label{lemma:HOT0-MLSchroEqu2018}
	Define $F_j(k,\hat x) ~(j=0,1)$ as in \eqref{eq:Fjkx-MLSchroEqu2018}.
	For every $\hat x \in \mathbb{S}^2$ and every $k_1, k_2 \geq k$, when $k \to +\infty$, we have the following estimates:
	\begin{alignat}{2}
	\big| \mathbb{E} \big( \overline{F_1(k_2,\hat x)}  F_0(k_1,\hat x) \big) \big| & = \mathcal{O}(k^{-m-1}),& \quad \big| \mathbb{E} \big( F_0(k_1,\hat x) \cdot F_1(k_2,\hat x) \big) \big| & = \mathcal{O}(k^{-m-1}), \label{eq:hotF0F1-MLSchroEqu2018}
	\end{alignat}
	uniformly for all $\hat x$.
\end{lem}

\begin{proof}
	We only prove the first asymptotic estimate in \eqref{eq:hotF0F1-MLSchroEqu2018} and the second one can be proved by following similar arguments. 
	For simplicity, we may use ${\mathcal D}_y$ to stand for $\mathcal D$ to indicate that the argument $y$ is integrated over this domain.

	In what follows we let $\hat x_1, \hat x_2 \in \mathbb{S}^2$. In this proof we may drop the arguments $k,\hat x$ if it is clear in the context. Write
	\begin{align}
	& G_0(k,\hat x) := \agl[f - \mathbb{E}f, e^{-ik \hat x \cdot (\cdot)}], \ G_j(k,\hat x) := \int e^{-ik \hat x \cdot y} \big( (V \Rk)^j (f - \mathbb{E}f) \big) (y) \dif{y}, \label{eq:Gjkx-MLSchroEqu2018} \\
	& r_j(k,\hat x) := \sum_{s \geq j} G_s(k,\hat x), \ \ j = 1,2,\cdots \label{eq:SumGjkx-MLSchroEqu2018}
	\end{align}
	thus $F_0 = G_0$ and $F_1 = r_1 = G_1 + r_2$, so we have
	\begin{equation} \label{eq:F0F1G0G1r2-MLSchroEqu2018}
	\mathbb{E} \big( F_0(k_1,\hat x_1) \cdot \overline{F_1(k_2,\hat x_2)} \big) = \mathbb{E} \big( G_0(k_1,\hat x_1) \cdot \overline{G_1(k_2,\hat x_2)} \,\big) + \mathbb{E} \big( G_0(k_1,\hat x_1) \cdot \overline{r_2(k_2,\hat x_2)} \big).
	\end{equation}
	
	To prove \eqref{eq:hotF0F1-MLSchroEqu2018}, we need to estimate $\mathbb{E} (G_0 \overline{G_1})$ and $\mathbb{E} (G_0 \overline{r_2})$. One can compute
	\begin{align}
	& \ | \mathbb{E} \big( G_0(k_1,\hat x_1) \cdot \overline{G_1(k_2,\hat x_2)} \,\big) | \nonumber\\
	= & \ \big| \mathbb{E} \big( \int_{{\mathcal D}_y} e^{-ik_1 \hat x_1 \cdot y} (f - \mathbb E f)(y) \dif{y} \times \overline{\int e^{-ik_2 \hat x_2 \cdot z} V(z) \int_{{\mathcal D}_t} \Phi(z,t) (f - \mathbb E f)(t) \dif{t} \dif{z}} \big) \big| \nonumber\\
	= & \ \big| \int e^{ik_2 \hat x_2 \cdot z} \overline V(z) \cdot \mathbb{E} \Big( \int_{{\mathcal D}_y} e^{-ik_1 \hat x_1 \cdot y} (f - \mathbb E f)(y) \dif{y} \cdot \int_{{\mathcal D}_t} \overline{\Phi}(z,t) (f - \mathbb E f)(t) \dif{t} \Big) \dif{z} \big| \nonumber\\
	= & \ \big| \int e^{ik_2 \hat x_2 \cdot z} \overline V(z) \cdot \big( \iint_{{\mathcal D} \times {\mathcal D}} K_f(t,y) e^{-ik_1 \hat x_1 \cdot y} \overline{\Phi}(z,t) \dif{y}\dif{t} \big) \dif{z} \big| \nonumber\\
	= & \ \big| \int e^{ik_2 \hat x_2 \cdot z} \overline V(z) \cdot \big( \int_ {\mathcal D} ( \mu(t)k_1^{-m} + a(t,-k_1 \hat x_1) ) e^{-ik_1 \hat x_1 \cdot t} \overline{\Phi}_{k_2}(z,t) \dif{t} \big) \dif{z} \big| \nonumber\\
	= & \ \nrm[L^1(\R^3)]{V \mathcal{R}_{k_2} (\mu(t) k_1^{-m} + \overline{a(\cdot,-k_1 \hat x_1)} e^{ik_1 \hat x_1 \cdot (\cdot)} \chi_{\mathcal D})} \nonumber \\
	\lesssim & \ k_2^{-1} \nrm[L_{1/2+\epsilon}^2(\R^3)]{(\mu(t) k_1^{-m} + \overline{a(\cdot,-k_1 \hat x_1)} e^{ik_1 \hat x_1 \cdot (\cdot)} \chi_{\mathcal D})} \nonumber \\
	\lesssim & \ k_2^{-1} k_1^{-m}, \quad k \to +\infty. \label{eq:hotG0G1-MLSchroEqu2018}
	\end{align}
	To estimate $\mathbb{E} \big( G_0(k_1,\hat x_1) \cdot \overline{r_2(k_2,\hat x_2)} \,\big)$ we first prove for $j > 1$,
	\begin{align} 
	\mathbb{E} \big( G_0(k_1,\hat x_1) \cdot \overline{G_j(k_2,\hat x_2)} \,\big) & = \overline{ \int e^{-ik_2 \hat x_2 \cdot z} (V \mathcal{R}_{k_2})^j \big( c_f(\cdot,k_1 \hat x_1) \,e^{ik_1 \hat x_1 \cdot (\cdot)} \,\chi_{\mathcal D} \big) (z) \dif{z} }, \label{eq:hotG0Gj-MLSchroEqu2018}\\
	\mathbb{E} \big( G_1(k_1,\hat x_1) \cdot \overline{G_j(k_2,\hat x_2)} \big) & = \overline{ \int e^{-ik_2 \hat x_2 \cdot z} \big( (V \mathcal{R}_{k_2})^j (\chi_{\mathcal D} \mathfrak C_f \overline{ \calR_{k_1}(V e^{-ik_1 \hat x_1 \cdot (\cdot)})}) \big)(z) \dif{z} }. \label{eq:hotG1Gj-MLSchroEqu2018}
	\end{align}
	We have
	\begin{align}
	& \mathbb{E} \big( \overline{G_0(k_1,\hat x_1)} \cdot G_j(k_2,\hat x_2) \big) \nonumber\\
	= & \mathbb{E} \big( \agl[f - \mathbb E f, e^{ik_1 \hat x_1 \cdot (\cdot)}] \cdot \int_{\mathcal D} e^{-ik_2 \hat x_2 \cdot z} (V \mathcal{R}_{k_2})^{j-1} \big( V(\cdot) \agl[(f - \mathbb E f)(y), \Phi_{k_2}(y,\cdot)] \big)(z) \dif{z} \big) \nonumber\\
	= & \int e^{-ik_2 \hat x_2 \cdot z} (V \mathcal{R}_{k_2})^{j-1} \Big( V(\cdot) \mathbb E \big( \agl[(f - \mathbb E f)(t), e^{ik_1 \hat x_1 \cdot t}] \agl[(f - \mathbb E f)(y), \Phi_{k_2}(y,\cdot)] \big) \Big)(z) \dif{z} \nonumber\\
	= & \int e^{-ik_2 \hat x_2 \cdot z} (V \mathcal{R}_{k_2})^j \big( c_f(\cdot,k_1 \hat x_1) e^{ik_1 \hat x_1 \cdot (\cdot)} \chi_{\mathcal D} \big) (z) \dif{z}. \label{eq:hotG0GjInter-MLSchroEqu2018}
	\end{align}
	By taking the conjugate of \eqref{eq:hotG0GjInter-MLSchroEqu2018}, we arrive at \eqref{eq:hotG0Gj-MLSchroEqu2018}. 
	Then to prove \eqref{eq:hotG1Gj-MLSchroEqu2018} one can compute
	{\small \begin{align}
		& \mathbb{E} \big( \overline{G_1(k_1,\hat x_1)} \cdot G_j(k_2,\hat x_2) \big) \nonumber\\
		= & \mathbb{E} \big( \int e^{ik_1 \hat x_1 \cdot x} \overline{ (V \mathcal{R}_{k_1} (f - \mathbb E f)) (x) } \dif{x} \cdot \int e^{-ik_2 \hat x_2 \cdot z} (V \mathcal{R}_{k_2})^j (f - \mathbb E f) (z) \dif{z} \big) \nonumber\\
		= & \int e^{-ik_2 \hat x_2 \cdot z} (V \mathcal{R}_{k_2})^{j-1} \Big( V(\cdot) \agl[(\mathfrak C_f \overline{ \calR_{k_1}(V e^{-ik_1 \hat x_1 \cdot (\cdot)})})(y), \chi_{\mathcal D}(y) \Phi_{k_2}(y,\cdot)] \Big)(z) \dif{z} \nonumber\\
		= & \int e^{-ik_2 \hat x_2 \cdot z} \big( (V \mathcal{R}_{k_2})^j (\chi_{\mathcal D} \mathfrak C_f \overline{ \calR_{k_1}(V e^{-ik_1 \hat x_1 \cdot (\cdot)})}) \big)(z) \dif{z}. \label{eq:hotG1GjInter-MLSchroEqu2018}
		\end{align}}
	We arrive at \eqref{eq:hotG1Gj-MLSchroEqu2018} by taking the conjugate of \eqref{eq:hotG1GjInter-MLSchroEqu2018}.
	By applying \eqref{eq:hotG0Gj-MLSchroEqu2018} we have
	\begin{align}
	& \big| \mathbb{E} \big( G_0(k_1,\hat x_1) \cdot \overline{r_2(k_2,\hat x_2)} \big) \big|
	\leq \sum_{j \geq 2} \big| \mathbb{E} \big( G_0(k_1,\hat x_1) \cdot \overline{G_j(k_2,\hat x_2)} \big) \big| \nonumber\\
	\leq & \sum_{j \geq 2} \nrm[L^1(\R^3)]{ (V \mathcal{R}_{k_2})^j \big( c_f(\cdot,k_1 \hat x_1) e^{ik_1 \hat x_1 \cdot (\cdot)} \chi_{\mathcal D} \big) } \nonumber\\
	\leq & C k_1^{-m} \cdot \sum_{j \geq 2} k_2^{-j} \nrm[L_{1/2+\epsilon}^2(\R^3)]{ k_1^mc_f(\cdot,k_1 \hat x_1) \chi_{\mathcal D}} \nonumber\\
	= & \mathcal{O}(k_1^{-m} k_2^{-2}), \quad k \to +\infty. \label{eq:hotG0r2-MLSchroEqu2018}
	\end{align}
	By \eqref{eq:F0F1G0G1r2-MLSchroEqu2018}, \eqref{eq:hotG0G1-MLSchroEqu2018} and \eqref{eq:hotG0r2-MLSchroEqu2018}, the formula \eqref{eq:hotF0F1-MLSchroEqu2018} is proved.
\end{proof}

Before we analyze the behavior of $\mathbb{E} \big( \overline{F_1(k_2,\hat x)} F_1(k_1,\hat x) \big)$ in terms of $k_1$ and $k_2$, we first present an auxiliary lemma that shall be useful in the proof of Lemma \ref{lemma:HOT1-MLSchroEqu2018}. In the sequel, we denote $\diam \Omega := \sup\limits_{x, x' \in \Omega} \{ |x - x'|\}$.

\begin{lem} \label{lem:intbdd-MLSchroEqu2018}
	Assume $\Omega$ is a bounded domain in $\Rn~(n \geq 1)$.
	For $\forall \alpha, \beta \in \R$ such that $\alpha < n$ and $\beta < n$, and for $\forall p \in \Rn \backslash \{0\}$, there exists a constant $C_{\alpha,\beta}$ independent of $p$ and $\Omega$ such that
	\begin{equation} \label{eq:intbdd-MLSchroEqu2018}
	\int_\Omega |t|^{-\alpha} |t - p|^{-\beta} \dif t 
	\leq C_{\alpha,\beta} \times
	\begin{cases} |p|^{n - \alpha - \beta} + (\diam  \Omega)^{n - \alpha - \beta}, & \alpha + \beta \neq n, \\
	\ln \frac 1 {|p|} + \ln (\diam \Omega) +C_{\alpha, \beta}, & \alpha + \beta = n.
	\end{cases}
	\end{equation}
\end{lem}

\begin{rem} \label{rem:intbdd-MLSchroEqu2018}
	Formula \eqref{eq:intbdd-MLSchroEqu2018} also holds when $p = 0$ and $\alpha + \beta \neq n$.
	When $p \neq 0$ and $\alpha + \beta \geq n$, the upper bound of the integral in \eqref{eq:intbdd-MLSchroEqu2018} goes to infinity as $p$ approaches the origin.
	When $p = 0$ and $\alpha + \beta \geq n$, the integral is ill-defined, i.e.~the Cauchy principal value of the integral is infinity. 
	Hence formula \eqref{eq:intbdd-MLSchroEqu2018} gives a description about how fast (in terms of $|p|$) the integral goes to infinity as $p$ approaches the origin.
\end{rem}

\begin{proof}[Proof of Lemma~\ref{lem:intbdd-MLSchroEqu2018}]
	We use $B(0,\diam \Omega)$ to signify the ball centering at the point $0$ and of radius $\diam(\Omega)$.
	We divide $\Omega$ into three parts:
	$\Omega_1 := B(p,|p|/2)$,
	$\Omega_2 := B(0,2|p|) \backslash \Omega_1$
	and $\Omega_3 := \Omega \backslash (\Omega_1 \cup \Omega_2)$.
	Noting that $\beta < n$, we can compute
	\begin{align} 
	\int_{\Omega_1} |t|^{-\alpha} |t - p|^{-\beta} \dif t
	& \leq \int_{\Omega_1} |p/2|^{-\alpha} |t - p|^{-\beta} \dif t
	= 2^{\alpha} |p|^{-\alpha} \int_{B(0,|p|/2)} |t|^{-\beta} \dif t \nonumber \\
	& = C_{\alpha, \beta} |p|^{n - \alpha - \beta}. \label{eq:intbdd1-MLSchroEqu2018}
	\end{align}
	Then we compute the integral over $\Omega_2$ as follows (noting that $\alpha < n$),
	\begin{align}
	\int_{\Omega_2} |t|^{-\alpha} |t - p|^{-\beta} \dif t
	& \leq \int_{\Omega_2} |t|^{-\alpha} |p/2|^{-\beta} \dif t
	= 2^\beta |p|^{-\beta} \int_{\Omega_2} |t|^{-\alpha} \dif t \nonumber \\
	& \leq 2^\beta |p|^{-\beta} \int_{B(0,2|p|)} |t|^{-\alpha} \dif t 
	= C_{\alpha,\beta} |p|^{n - \alpha - \beta}. \label{eq:intbdd2-MLSchroEqu2018}
	\end{align}
	
	We claim that $|t|/2 \leq |t - p| \leq 3|t|/2$ for $\forall t \in \Omega_3$. 
	This can be seen in the following way: 
	fix a quantity $T > 2|p|$, then $p$ is an inner point of the ball $B(0,T)$.
	The distance between $t$ and $p$ is $|t - p|$.
	For every $t$ such that $|t| = T$, 
	the longest distance between $t$ and $p$ is $T + |p|$ while the shortest distance is $T - |p|$, thus $T - |p| \leq |t - p| \leq T + |p|$ holds.
	Because $T > 2|p|$ and $|t| = T$, we obtain $|t|/2 \leq |t - p| \leq 3|t|/2$ for $\forall t \in \Omega_3$.
	The quantity $\diam \Omega$ is finite because $\Omega$ is bounded.
	Therefore, the integral over $\Omega_3$ can be  computed as follows,
	\begin{align} 
	\int_{\Omega_3} |t|^{-\alpha} |t - p|^{-\beta} \dif t
	& \leq \int_{\Omega_3} |t|^{-\alpha} (|t|/2)^{-\beta} \dif t
	\leq 2^{|\beta|} \int_{ \{2|p| \leq |t| \leq \diam \Omega\} } |t|^{-\alpha-\beta} \dif t \nonumber \\
	& \leq
	\begin{cases}
	\frac {2^{|\beta|}} {n - \alpha - \beta} [(\diam  \Omega)^{n - \alpha - \beta} - |p|^{n - \alpha - \beta}], & \alpha + \beta \neq n, \\
	2^{|\beta|} [\ln \frac 1 {|p|} - \ln 2 + \ln (\diam \Omega)], & \alpha + \beta = n,
	\end{cases} \nonumber \\
	& \leq C_{\alpha,\beta} \times
	\begin{cases} |p|^{n - \alpha - \beta} + (\diam  \Omega)^{n - \alpha - \beta}, & \alpha + \beta \neq n, \\
	\ln \frac 1 {|p|} + \ln (\diam \Omega) - \ln 2, & \alpha + \beta = n.
	\end{cases}
	\label{eq:intbdd3-MLSchroEqu2018}
	\end{align}
	Summing up \eqref{eq:intbdd1-MLSchroEqu2018}, \eqref{eq:intbdd2-MLSchroEqu2018} and \eqref{eq:intbdd3-MLSchroEqu2018}, we obtain \eqref{eq:intbdd-MLSchroEqu2018}.
	The proof is complete. 
\end{proof}

\begin{lem} \label{lemma:HOT1-MLSchroEqu2018}
	Define $F_j(k,\hat x) ~(j=0,1)$ as in \eqref{eq:Fjkx-MLSchroEqu2018}. For every $\hat x \in \mathbb{S}^2$ and every $k_1, k_2 \geq k$, when $k \to +\infty$, we have the following estimates:
	\begin{alignat}{2}
	\big| \mathbb{E} \big( \overline{F_1(k_2,\hat x)} F_1(k_1,\hat x) \big) \big| & = \mathcal{O}(k^{-3}),& \quad \big| \mathbb{E} \big( F_1(k_1,\hat x) \cdot F_1(k_2,\hat x) \big) \big| & = \mathcal{O}(k^{-3}), \label{eq:hotF1F1-MLSchroEqu2018}
	\end{alignat}
	uniformly for all $\hat x$.
\end{lem}

\begin{proof}
	We only prove the first asymptotic estimate in \eqref{eq:hotF1F1-MLSchroEqu2018} and the second one can be proved by following  similar arguments.		
	We continue to use the notation $G_j$ defined in \eqref{eq:Gjkx-MLSchroEqu2018}.
	To prove the statement, the following two identities are needed:
	\begin{equation} \label{eq:hotGj-MLSchroEqu2018}
	G_j(k,\hat x) = \agl[(f - \mathbb E f)(s), \int e^{-ik \hat x \cdot y} {\big[} (V\Rk)^{j-1} (V(\cdot) \Phi(s,\cdot)) {\big]} (y) \dif{y}] \quad (j \geq 1),
	\end{equation}
	{\small \begin{align} \label{eq:hotGjGl-MLSchroEqu2018}
		& \ \mathbb{E} \big( G_j(k_1,\hat x_1) \cdot \overline{G_\ell(k_2,\hat x_2)} \big) \nonumber\\
		= & \int e^{ik_2 \hat x_2 \cdot z} \bigg\{ (V \mathcal{R}_{k_2})^{\ell-1} \Big( \int e^{-ik_1 \hat x_1 \cdot y} \big[ (V \mathcal{R}_{k_1})^{j-1} (V(1) \overline V (2) I(2,1)) \big] (y) \dif{y} \Big) \bigg\}(z) \dif{z} \ (j,\ell \geq 1),
		\end{align}}
	where the operation $\agl[\cdot,\cdot]$ in \eqref{eq:hotGj-MLSchroEqu2018} is in terms of the variable $s$, and
	\begin{equation} \label{eq:DefI-MLSchroEqu2018}
	I(x,y) := \iint_{D_f \times D_f} K_f(s,t) \Phi(s-y) \overline \Phi(t-x) \dif{s} \dif{t}.
	\end{equation}
	In \eqref{eq:hotGjGl-MLSchroEqu2018}, with some abuse of notations, we use ``1'' (resp.~``2'') to represent the variable that the operator $V \calR_{k_1}$ (resp.~$V \calR_{k_2}$) acts on.
	
	To prove \eqref{eq:hotGj-MLSchroEqu2018}, one can compute
	\begin{align} \label{eq:VRkjf-MLSchroEqu2018}
	[(V \Rk)^j f](x)
	& = [(V \Rk)^{j-1}((V \Rk)f)](x) = \big[ (V \Rk)^{j-1}( V(\cdot) \agl[f(s), \Phi_k(s,\cdot)] ) \big] (x) \nonumber\\
	& = \agl[f(s), {\big[} (V\Rk)^{j-1} (V(\cdot) \Phi(s,\cdot)) {\big]} (x)].
	\end{align}
	By \eqref{eq:Gjkx-MLSchroEqu2018} and \eqref{eq:VRkjf-MLSchroEqu2018}, we arrive at \eqref{eq:hotGj-MLSchroEqu2018}.
	To prove \eqref{eq:hotGjGl-MLSchroEqu2018}, one can compute
	\begin{align*}
	& \ \mathbb{E} \big( G_j(k_1,\hat x_1) \cdot \overline{G_\ell(k_2,\hat x_2)} \big) \\
	= & \ \mathbb{E} \Big( \Agl[(f - \mathbb E f)(s), \int e^{-ik_1 \hat x_1 \cdot y} {\big[} (V\mathcal{R}_{k_1})^{j-1} (V(\cdot) \Phi(s,\cdot)) {\big]} (y) \dif{y}] \nonumber\\
	& \ \ \ \cdot \Agl[(f - \mathbb E f)(t), \int e^{ik_2 \hat x_2 \cdot z} {\big[} (V\mathcal{R}_{k_2})^{\ell-1} (\overline V(\cdot) \overline \Phi(t,\cdot)) {\big]} (z) \dif{z}] \Big) \\
	= & \iint_{{D_f} \times {D_f}} \int e^{-ik_1 \hat x_1 \cdot y} {\big[} (V\mathcal{R}_{k_1})^{j-1} (K(s,t) V(\cdot) \Phi(s,\cdot)) {\big]} (y) \dif{y} \nonumber\\
	& \ \ \cdot \int e^{ik_2 \hat x_2 \cdot z} {\big[} (V\mathcal{R}_{k_2})^{\ell-1} (\overline V(\cdot) \overline \Phi(t,\cdot)) {\big]} (z) \dif{z} \dif{s} \dif{t} \\
	= & \int e^{ik_2 \hat x_2 \cdot z} \bigg\{ (V \mathcal{R}_{k_2})^{\ell-1} \Big( \int e^{-ik_1 \hat x_1 \cdot y} \big[ (V \mathcal{R}_{k_1})^{j-1} (V(1) \overline V (2) I(2,1)) \big] (y) \dif{y} \Big) \bigg\}(z) \dif{z}.
	\end{align*}
	Thus, \eqref{eq:hotGjGl-MLSchroEqu2018} is proved.
	
	\smallskip

	Note that
	{\small \begin{align}
		\mathbb{E} \big( F_1(k_1,\hat x_1) \cdot \overline{F_1(k_2,\hat x_2)} \big) & = \mathbb{E} \big( G_1(k_1,\hat x_1) \cdot \overline{G_1(k_2,\hat x_2)} \,\big) + \sum_{\substack{j+\ell \geq 3\\j,\ell \geq 1}} \mathbb{E} \big( G_j(k_1,\hat x_1) \cdot \overline{G_\ell(k_2,\hat x_2)} \,\big). \label{eq:F1F1G1GjGl-MLSchroEqu2018}
		\end{align}}
	Next we estimate $\mathbb{E} (G_1 \overline{G_1})$ and $\mathbb{E} (G_j \overline{G_\ell}) \,(j+\ell \geq 3, \,j,\ell \geq 1)$ in different manners.
	
	Recall the definition of $\mathcal D$ given in \eqref{eq:calD-MLSchroEqu2018}.
	We denote $\widetilde{\mathcal D} := \{x + x', x - x' \,;\, x, x' \in \mathcal D \}$.
	To estimate $\mathbb{E} (G_1 \overline{G_1})$, we fix real-valued cut-off functions $\eta_i \in \mathcal{C}_c^\infty(\R^3) \,(i=1,2)$ satisfying
	\begin{equation} \label{eq:etaj-MLSchroEqu2018}
	\begin{cases}
	\supp \eta_i \subset \widetilde{\mathcal D}, \ i = 1,2, \\
	\eta_1 = 1 \text{~in~} D_f, \\
	\eta_2 = 1 \text{~in~} \{s+t \in \R^3 \,;\, s,t \in D_f \}.
	\end{cases}
	\end{equation}
	With the help of \eqref{eq:hotGjGl-MLSchroEqu2018} and \eqref{eq:KtoSymbol-MLSchroEqu2018} and by using \cite[Lemma 18.2.1]{hormander1985analysisIII} repeatedly, one have
	{\small \begin{align}
		& \mathbb{E} \big( G_1(k_1,\hat x_1) \cdot \overline{G_1(k_2,\hat x_2)} \big) \nonumber\\
		= & \int e^{ik_2 \hat x_2 \cdot z} \int e^{-ik_1 \hat x_1 \cdot y} V(y) \overline V (z) I(z,y) \dif{y} \dif{z} \nonumber\\
		\simeq & \iint \eta_2(s+t) \eta_1(s) \eta_1(t) \big( \int e^{i (s-t) \cdot \xi} c_f(s,\xi) \dif \xi \big) \cdot \big( \int e^{-ik_1 (\hat x_1 \cdot y - |y-s|)} \frac {V(y)} {|y-s|} \dif y \big) \nonumber\\
		& \quad \cdot \big( \int e^{ik_2 (\hat x_2 \cdot z - |z-t|)} \frac {\overline V(z)} {|z-t|} \dif z \big) \dif s \dif t \nonumber\\
		= & \iint \eta_2(s+t) \big( \int e^{i (s-t) \cdot \xi} \tilde c(s,t,\xi) \dif \xi \big) e^{-ik_1 \hat x_1 \cdot s} e^{ik_2 \hat x_2 \cdot t} {\mathcal G} (s,k_1,\hat x_1) \overline {\mathcal G} (t,k_2,\hat x_2) \dif s \dif t \nonumber\\
		= & \frac 1 2 \iint \eta_2(T) e^{i\theta_2 \cdot T} e^{-i\theta_1 \cdot S} \big( \int e^{i S \cdot \xi} {\mathcal G}(\frac {T+S} 2,k_1,\hat x_1) \overline {\mathcal G}(\frac {T-S} 2,k_2,\hat x_2) c_2(T,\xi) \dif \xi \big) \dif S \dif T \nonumber\\
		= & \frac 1 2 \iint \eta_2(T) e^{i\theta_2 \cdot T} e^{-i\theta_1 \cdot S} \big( \int e^{i S \cdot \xi} \tilde c_3(S,T,\xi) \dif \xi \big) \dif S \dif T \nonumber\\
		= & \frac 1 2 \int \eta_2(T) e^{i \theta_2 \cdot T} \Big( \int e^{-i \theta_1 \cdot S} \big( \int e^{i S \cdot \xi} c_3(T,\xi) \dif \xi \big) \dif S \Big) \dif T \nonumber\\
		\simeq & \int_{\R^3} \eta_2(T) e^{i\theta_2 \cdot T} c_3(T,\theta_1) \dif T, \label{eq:hotG1G1Inter1-MLSchroEqu2018}
		\end{align}}
	where
	\[
	{\mathcal G}(s,k,\hat x) := \int_{\R^3} e^{-ik (\hat x \cdot y - |y|)} \frac {V(y+s)} {|y|} \dif y ,
	\]
	and
	\[
	\begin{cases}
	\theta_1 := (k_1 \hat x_1 + k_2 \hat x_2)/2 \\
	\theta_2 := (k_1 \hat x_1 - k_2 \hat x_2)/2
	\end{cases}
	\quad \text{and} \quad
	\begin{cases}
	S := s-t \\
	T := s+t
	\end{cases},
	\]
	and
	\begin{equation*}
	\begin{cases}
	& \tilde c(s,t,\xi) := \eta_1(s) \eta_1(t) c_f(s,\xi), \\
	& c_2(T,\xi) = \tilde c(T/2,T/2,\xi) + S^{-m-1} = (\eta_1(T/2))^2 c(T/2,\xi) + S^{-m-1}, \\
	& \tilde c_3(S,T,\xi) = \tilde c_3(S,T,\xi;k_1,\hat x_1,k_2,\hat x_2) := {\mathcal G}(\frac {T+S} 2,k_1,\hat x_1) \,\overline {\mathcal G}(\frac {T-S} 2,k_2,\hat x_2) \,c_2(T,\xi), \\
	& c_3(T,\xi) = \tilde c_3(0,T,\xi) + S^{-m-1}.
	\end{cases}
	\end{equation*}
	Here the notation $S^{-m-1}$ stands for the set of symbols of pseudo-differential operators of order $-m-1$; see e.g. \cite{wong2014pdo} for more details about pseudo-differential operators.
	Therefore,
	\begin{align}
	c_3(T,\xi)
	& = {\mathcal G}(T/2,k_1,\hat x_1)\, \overline {\mathcal G}(T/2,k_2,\hat x_2)\, c_2(T,\xi) + S^{-m-1} \nonumber\\
	& = (\eta_1(T/2))^2\, {\mathcal G}(T/2,k_1,\hat x_1) \,\overline {\mathcal G}(T/2,k_2,\hat x_2) \,c(T/2,\xi) \nonumber\\
	& \quad + {\mathcal G}(T/2,k_1,\hat x_1) \,\overline {\mathcal G}(T/2,k_2,\hat x_2) \cdot S^{-m-1} \nonumber\\
	& = (\eta_1(T/2))^2\, {\mathcal G}(T/2,k_1,\hat x_1) \,\overline {\mathcal G}(T/2,k_2,\hat x_2) \,c(T/2,\xi) + S^{-m-1}. \label{eq:c3-MLSchroEqu2018}
	\end{align}
	Set $\hat x = \hat x_1 = \hat x_2$ and recall that $|\mathcal S|$ signifies the Lebesgue measure of any Lebesgue-measurable set $\mathcal S$, from \eqref{eq:hotG1G1Inter1-MLSchroEqu2018} and \eqref{eq:c3-MLSchroEqu2018} we obtain
	\begin{align}
	& \ |\mathbb{E} \big( G_1(k_1,\hat x) \cdot \overline{G_1(k_2,\hat x)} \big)| \leq C |\supp \eta_2| \cdot \sup_{T \in \supp \eta_3} |c_3(T,\theta_1)| \nonumber\\
	\leq & \ C |\supp \eta_2| \agl[\theta_1]^{-m} \big( \sup_{T \in \supp \eta_2} |{\mathcal G}(T/2,k_1,\hat x)| \cdot |\overline {\mathcal G}(T/2,k_2,\hat x)| + C |\supp \eta_2| \cdot \agl[\theta_1]^{-1} \big) \nonumber\\
	\leq & \ C_f \sup_{T \in \supp \eta_2} |{\mathcal G}(T/2,k_1,\hat x)| \cdot |\overline {\mathcal G}(T/2,k_2,\hat x)| \cdot k^{-m} + C_f k^{-m-1}, \label{eq:hotG1G1Inter2-MLSchroEqu2018}
	\end{align}
	where the constant $C_f$ is independent of $k$, $k_1$, $k_2$ and $\hat x$.
	
	We proceed to show that ${\mathcal G}(T/2,k,\hat x) = \mathcal{O}(k^{-1})$. For any $\hat x \in \mathbb{S}^2$, we can always find two unit vectors $\hat x^{\perp,1}, \hat x^{\perp,2} \in \mathbb{S}^2$ such that the set $\{\hat x, \hat x^{\perp,1}, \hat x^{\perp,2} \}$ forms an orthonormal basis. 
	Write the $3 \times 3$ matrix $\Phi = (\hat x,\hat x^{\perp,1}, \hat x^{\perp,2})$, then $\Phi^T \hat x = (1,0,0)^T =: e_1$. 
	Denoting
	\[
	\widetilde V(y,s) := \agl[y]^{1+\sigma} V(y+s),
	\]
	where the value of $\sigma$ shall be determined later,
	we know $\widetilde V(y,s) \in C^3$ in $y$ variable.
	We have
	\begin{align*}
	{\mathcal G}(s,k,\hat x)
	& = \int_{\R^3} e^{-ik (\hat x \cdot y - |y|)} |y|^{-1} \agl[y]^{-1-\sigma} \widetilde V(y,s) \dif y \\
	& = \mathcal{O}(k^{-1}) + \int_{k^{-1/2}}^{+\infty} r \agl[r]^{-1-\sigma} e^{ikr} \dif r \cdot \int_{\mathbb{S}^2} e^{ikr \hat x \cdot w} \widetilde V(rw, s) \dif{S(w)} \\
	& = \mathcal{O}(k^{-1}) + \int_{k^{-1/2}}^{+\infty} r \agl[r]^{-1-\sigma} e^{ikr} \dif r \cdot \int_{\mathbb{S}^2} e^{ikr e_1 \cdot w} \widetilde V(r \Phi w, s) \dif{S(w)}, \quad k \to +\infty.
	\end{align*}
	We cover the  unit sphere $\mathbb{S}^2$ by six (relative) open parts:
	$$\Gamma_{p,q} := \{ (w_1,w_2,w_3) \in \R^3 ; \sum_{j=1}^3 w_j^2 = 1, (-1)^q w_p > \sqrt{3}/6 \}, ~p=1,2,3, ~q=0,1.$$
	It is straightforward to verify that $\{\Gamma_{p,q}\}$ is an open covering of $\mathbb{S}^2$, i.e.~$\mathbb{S}^2 \subset \cup_{p,q} \Gamma_{p,q}$. 
	There exists a partition of unity $\{\rho_{p,q}\}$ subject to the open covering $\{\Gamma_{p,q}\}$, and we write
	$$g_{p,q}(r,k,\hat x,s) := \int_{\Gamma_{p,q}} e^{ikr e_1 \cdot w} \rho_{p,q}(w) \widetilde V(r \Phi w, s) \dif{S(w)}.$$
	Hence,
	\begin{equation} \label{eq:fsInter-MLSchroEqu2018}
	\mathcal G (s,k,\hat x) = \mathcal{O}(k^{-1}) + \sum_{p,q} \int_{k^{-1/2}}^{+\infty} r \agl[r]^{-1-\sigma} e^{ikr} g_{p,q}(r,k,\hat x,s) \dif r.
	\end{equation}
	We proceed to analyze $g_{1,0}$ and $g_{3,0}$. The analysis of $g_{1,1}$ is similar to that of $g_{1,0}$, and $g_{p,q} \,(p=2,3,\, q=0,1)$ is similar to $g_{3,0}$, so we skip the analyses of these terms.
	
	In what follows, we write $w = (w_1,w_2,w_3)^T \in \mathbb{S}^2$ as a vertical vector. 
	Noticing that in $\Gamma_{1,0}$ the $w_1$ is uniquely determined by the $w_2$ and $w_3$, so there exists a unique function $\phi \in C^\infty$ such that $w_1 = \phi(w_2,w_3)$, and with a slight abuse of notation, we may write $w = w(w_1,w_2) = (\phi(w_2,w_3),w_2,w_3)^T$. 
	Denote the projection of $\Gamma_{1,0}$ onto the $(w_2,w_3)$-coordinate as $\Pi_{1,0}$. We know $\Pi_{1,0} \subset (-1,1)^2$.  We have
	$$\phi(w_2,w_3) \in (\sqrt{30}/6,1], \Forall (w_2,w_3) \in \Pi_{1,0}.$$
	We can fix some $\rho_{1,0} \in \mathcal{C}_c^\infty((-1,1)^2)$ such that $\rho_{1,0} \equiv 1$ in $\Pi_{1,0}$. Then
	\begin{equation} \label{eq:g10kr-MLSchroEqu2018}
	\begin{split}
	g_{1,0}
	& = \int_{\R^2} e^{ikr \phi(w_2,w_3)} \rho_{1,0}(w_2,w_3) \widetilde V(r \Phi w, s)  \\
	& \ \ \ \cdot \sqrt{\det [ (\partial_{w_2} w, \partial_{w_3} w)^T (\partial_{w_2} w, \partial_{w_3} w)]} \dif{w_2} \dif{w_3}.
	\end{split}		
	\end{equation}
	According to $\phi^2 + w_2^2 + w_3^2 = 1$ we have
	\begin{equation*}
	\begin{cases}
	\phi_{w_2} = -w_2/\phi \\
	\phi_{w_3} = -w_3/\phi
	\end{cases}
	\text{and}\quad
	\begin{cases}
	\phi_{w_2 w_2} = -(1+\phi_{w_2}^2)/\phi \\
	\phi_{w_2 w_3} = -\phi_{w_2} \phi_{w_3}/\phi \\
	\phi_{w_3 w_3} = -(1+\phi_{w_3}^2)/\phi
	\end{cases}.
	\end{equation*}
	Note that $\phi > \sqrt{30}/6$. Hence, we have that $|\nabla \phi| = 0$ only when $w_2 = w_3 = 0$ and that $\det [\frac {\partial_2 \phi} {\partial w_2 \partial w_3}] = (1 + \phi_{w_2}^2 + \phi_{w_3}^2)/\phi^2 \neq 0$. 
	This means that $(0,0)$ is the only critical point of the phase function $kr\phi(w_2,w_3)$ in \eqref{eq:g10kr-MLSchroEqu2018}, when $w \in \Gamma_{1,0}$. 
	According to the stationary phase lemma \cite[Chapter 3]{zw2012semi},
	we have
	\begin{equation}
	g_{1,0}(r,k,\hat x,s)
	= \Big( \frac {2\pi} {kr} \Big) C_1 \left( C_2 + \mathcal{O}((kr)^{-1})\right) = \mathcal{O}((kr)^{-1}), \quad k \to +\infty. \label{eq:g10-MLSchroEqu2018}
	\end{equation}
	Note that in order to use the stationary phase lemma to obtain the high-order term with $-1$ order decay, the integrand should have $C^5$-smoothness, which is guaranteed by \eqref{eq:asp-MLSchroEqu2018}.
	
	Next we analyze $g_{3,0}$. We may write $w = w(w_1,w_2) = (w_1,w_2,\phi(w_1,w_2))^T$.
	It holds that
	\begin{align*}
	& g_{3,0}(r,k,\hat x,s) \\
	= & \int_{\R^2} e^{ikr w_1} \rho_{3,0}^2(w_1,w_2) \widetilde V(r \Phi w, s) \nonumber\\
	& \ \ \cdot \sqrt{\det [ (\partial_{w_1} w, \partial_{w_2} w)^T (\partial_{w_1} w, \partial_{w_2} w)]} \dif{w_1} \dif{w_2} \\
	= & \frac 1 {ikr} \int_{\R^2} \partial_{w_1} (e^{ikr w_1}) \rho_{3,0}(w) \widetilde V(r \Phi w, s) \mathcal C_1(w_1,w_2) \dif{w_1} \dif{w_2} \\
	= & \frac i {kr} \int_{\R^2} e^{ikr w_1} \partial_{w_1} \big( \mathcal C_2(w_1,w_2;|\hat x|,V) \big) \dif{w_1} \dif{w_2} ,
	\end{align*}
	where $\mathcal C_1$ and $\mathcal C_2 := \rho_{3,0}(w) \widetilde V(r \Phi w, s) \mathcal C_1(w_1,w_2)$ are two functions such that $\mathcal C_1 \in C^\infty$ and $\mathcal C_2 \in C_c^3((-1,1)^2)$ because $\widetilde V(\cdot,s) 
	\in C^3$, and $\rho_{3,0}$ is chosen in the same manner as $\rho_{1,0}$. 
	Therefore the partial derivative of the function $\mathcal C_2$ is bounded above, and hence
	\begin{equation}  \label{eq:g30-MLSchroEqu2018}
	|g_{3,0}(r,k,\hat x,s)| \leq C (kr)^{-1}.
	\end{equation}
	
	Combining \eqref{eq:fsInter-MLSchroEqu2018} with \eqref{eq:g10-MLSchroEqu2018} and \eqref{eq:g30-MLSchroEqu2018}, one can compute
	\begin{align}
	|\mathcal G(s,k,\hat x)|
	& \leq \mathcal{O}(k^{-1}) + \sum_{p,q} \int_{k^{-1/2}}^{+\infty} r^{-\sigma} [C_1 (kr)^{-1} + C_2 (kr)^{-2} + \mathcal{O}((kr)^{-3}) ] \dif r \nonumber \\
	& = \mathcal{O}(k^{-1+\sigma/2}), \quad k \to +\infty, \label{eq:fs-MLSchroEqu2018}
	\end{align}
	where the asymptotics is uniform in terms of $s$ and $\hat x$.
	Note that we used $r \agl[r]^{-1-\sigma} \leq r^{-\sigma}$.
	By \eqref{eq:hotG1G1Inter2-MLSchroEqu2018} and \eqref{eq:fs-MLSchroEqu2018}, we arrive at
	\begin{equation*}
	|\mathbb{E} \big( G_1(k_1,\hat x_1) \cdot \overline{G_1(k_2,\hat x_2)} \big)| \leq C k^{-m-2+\sigma} + C \cdot k^{-m-1} = \mathcal{O}(k^{-m-1}), ~k \to +\infty,
	\end{equation*}
	where the last inequality is by taking $\sigma = 1/3$.
	Because $m > 2$, we obtain
	\begin{equation} \label{eq:hotG1G1-MLSchroEqu2018}
	|\mathbb{E} \big( G_1(k_1,\hat x_1) \cdot \overline{G_1(k_2,\hat x_2)} \big)| \leq \mathcal{O}(k^{-3}), ~k \to +\infty.
	\end{equation}

	\smallskip
	
	To estimate $\mathbb{E} (G_j \overline{G_\ell})$ for $j + \ell \geq 3$, $j,\,\ell \geq 1$, we first estimate $I(z,y)$ which is defined in \eqref{eq:DefI-MLSchroEqu2018}. 
	Choose $\eta_1, \eta_2 \in \mathcal{C}_c^\infty(\R^3)$ as in \eqref{eq:etaj-MLSchroEqu2018}.
	It follows that
	\begin{align}
	& I(z,y) \nonumber\\
	= \ & \iint_{{D_f} \times {D_f}} K(s,t) \eta_1(s) \eta_1(t) \Phi(s-y) \overline \Phi(t-z) \dif{s} \dif{t} \nonumber\\
	\simeq \ & \iint_{{D_f} \times {D_f}} \calF^{-1} \big\{ c(s,\cdot) \big\} (s-t) \cdot \eta_1(s) \eta_1(t) \Phi(s-y) \overline \Phi(t-z) \dif{s} \dif{t} \nonumber\\
	\simeq \ & \iint_{{D_f} \times {D_f}} e^{ik_1|s-y| -ik_2|t-z|} \big( |s-y|^{-1} |t-z|^{-1} \int e^{i(s-t) \cdot \xi} c_1(s,t,\xi) \dif{\xi} \big)  \dif{s} \dif{t}, \label{eq:I.inter1-MLSchroEqu2018}
	\end{align}
	where $c_1(s,t,\xi) := c(s,\xi) \eta_1(s) \eta_1(t)$. 
	Define two differential operators 
	\[
	L_1 := \frac {(s-y) \cdot \nabla_s} {ik_1|s-y|}
	\quad \text{and} \quad
	L_2 := \frac {(t-z) \cdot \nabla_t} {-ik_2|t-z|}.
	\]
	It can be  verified that
	\[
	L_1 L_2 (e^{ik_1|s-y|-ik_2|t-z|}) = e^{ik_1|s-y|-ik_2|t-z|}.
	\]
	Hence, noting that the integrand is compactly supported in ${D_f} \times {D_f}$ and by using integration by part, we can continue \eqref{eq:I.inter1-MLSchroEqu2018} as
	\begin{align}
	& \ |I(z,y)| \nonumber\\
	\simeq & \ |\iint_{{D_f} \times {D_f}} L_1 L_2 (e^{ik_1|s-y| -ik_2|t-z|}) \big( |s-y|^{-1} |t-z|^{-1} \int e^{i(s-t) \cdot \xi} c_1(s,t,\xi) \dif{\xi} \big)  \dif{s} \dif{t}| \nonumber\\
	\simeq & \ k_1^{-1} k_2^{-1} |\iint_{{D_f} \times {D_f}} e^{ik_1|s-y| -ik_2|t-z|} \nonumber \\
	& \hspace*{1.5cm} \times \Big\{ \divr \big( \frac {s-y} {|s-y|} \big) |s-y|^{-1} \big[ \divr \big( \frac {t-z} {|t-z|} \big) |t-z|^{-1} \int e^{(s-t) \cdot \xi} c_1 \dif \xi \nonumber \\
	& \hspace*{6cm} + \frac {t-z} {|t-z|^2} \cdot \nabla_t \int e^{(s-t) \cdot \xi} c_1 \dif \xi \big] \nonumber \\
	& \hspace*{2.3cm} + \frac {s-y} {|s-y|^2} \cdot \big[ \divr \big( \frac {t-z} {|t-z|} \big) |t-z|^{-1} \nabla_s \int e^{(s-t) \cdot \xi} c_1 \dif \xi \nonumber \\
	& \hspace*{6cm} + \frac {t-z} {|t-z|^2} \cdot \nabla_t \nabla_s \int e^{(s-t) \cdot \xi} c_1 \dif \xi \big] \Big\} \dif s \dif t | \nonumber \\
	\lesssim & \ k_1^{-1} k_2^{-1} \iint_{{D_f} \times {D_f}} \big[ |s-y|^{-2} |t-z|^{-2} \mathcal J_0 + |s-y|^{-2} |t-z|^{-1} (\max_a \mathcal J_{1;a}) \nonumber \\
	& \hspace*{2cm} + |s-y|^{-1} |t-z|^{-2} (\max_{a}\mathcal J_{1;a}) + |s-y|^{-1} |t-z|^{-1} (\max_{a,b} \mathcal J_{2;a,b}) \big] \dif s \dif t, \label{eq:I.inter2-MLSchroEqu2018}
	\end{align}
	where $a,b$ are indices running from 1 to 3, and
	\begin{align*}
	\mathcal J_0
	& := |\int e^{i(s-t) \cdot \xi}\, c_1(s,t,\xi) \dif \xi|, \quad
	\mathcal J_{1;a}
	:= |\int e^{i(s-t) \cdot \xi}\, \xi_a c_1(s,t,\xi) \dif \xi|, \\
	\mathcal J_{2;a,b}
	& := |\int e^{i(s-t) \cdot \xi}\, \xi_a \xi_b c_1(s,t,\xi) \dif \xi|.\end{align*}

	Because of the condition $m > 2$, we can find a  number $\tau \in (0,1)$ satisfying the inequalities $3 - m < \tau < 1$.
	Therefore, we have
	\begin{subequations} \label{eq:hotF1F1J3T-MLSchroEqu2018}
		\begin{numcases}{}
		- m - \tau < -3, \label{eq:hotF1F1J3s1-MLSchroEqu2018} \\
		- 2 - \tau > -3. \label{eq:hotF1F1J3s2-MLSchroEqu2018}
		\end{numcases}
	\end{subequations}
	By using \cite[Lemmas 3.1 and 3.2]{HLSM2019both}, these quantities $\mathcal J_0$, $\mathcal J_{1;a}$ and $\mathcal J_{2;a,b}$ can be estimated as follows:	
	\begin{align}
	\mathcal J_0
	& = |s-t|^{-\tau} \cdot |\int (-\Delta_\xi)^{\tau/2} (e^{i(s-t) \cdot \xi}) c_1 (s,t,\xi) \dif \xi| \nonumber \\
	& = |s-t|^{-\tau} \cdot |\int e^{i(s-t) \cdot \xi}\, (-\Delta_\xi)^{\tau/2} (c_1 (s,t,\xi)) \dif \xi| \nonumber \\
	& \lesssim |s-t|^{-\tau} \cdot \int \agl[\xi]^{-m-\tau} \dif \xi  \lesssim |s-t|^{-\tau}. \label{eq:hotF1F1J1-MLSchroEqu2018}
	\end{align}
	The last inequality in \eqref{eq:hotF1F1J1-MLSchroEqu2018} makes use of the fact \eqref{eq:hotF1F1J3s1-MLSchroEqu2018}.
	We estimate $\mathcal J_{1;a}$ as follows,
	\begin{align}
	\mathcal J_{1;a}
	& = |\int \frac {(s - t) \cdot \nabla_{\xi}} {i|s-t|^{2+\tau}} ((-\Delta_\xi)^{\tau/2} e^{i(s-t) \cdot \xi})\, \xi_a c_1(s,t,\xi) \dif \xi| \nonumber\\
	& = \frac {|s - t|} {|s-t|^{2+\tau}} |\int e^{i(s-t) \cdot \xi}\, (-\Delta_\xi)^{\tau/2} \big(\nabla_\xi ( \xi_a c_1(s,t,\xi) ) \big) \dif \xi| \nonumber\\
	& \leq C |s-t|^{-1-\tau} \int \agl[\xi]^{-m+1-1-\tau} \dif \xi
	\leq C |s-t|^{-1-\tau}, \label{eq:hotF1F1J2-MLSchroEqu2018}
	\end{align}
	where the constant $C$ is independent of the index $a$.
	Similarly, we have
	\begin{align}
	\mathcal J_{2;a,b}
	& = |s - t|^{-2-\tau} |\int \Delta_\xi (-\Delta_\xi)^{\tau/2} (e^{i(s-t) \cdot \xi})\, \xi_a \xi_b c_1(s,t,\xi) \dif \xi| \nonumber \\
	& \leq C |s - t|^{-2-\tau} |\int  \agl[\xi]^{-m+2-2-\tau} \dif \xi|
	\leq C |s - t|^{-2-\tau}, \label{eq:hotF1F1J3-MLSchroEqu2018}
	\end{align}
	where the constant $C$ is independent of the indices $a$ and $b$.
	Combining \eqref{eq:I.inter2-MLSchroEqu2018}, \eqref{eq:hotF1F1J1-MLSchroEqu2018}, \eqref{eq:hotF1F1J2-MLSchroEqu2018} and \eqref{eq:hotF1F1J3-MLSchroEqu2018}, we can rewrite \eqref{eq:I.inter2-MLSchroEqu2018} as
	\begin{align}
	k_1 k_2 |I(z,y)|
	& \lesssim \iint_{{D_f} \times {D_f}} \big[ |s-y|^{-2} |t-z|^{-2} |s - t|^{-\tau} + |s-y|^{-2} |t-z|^{-1} |s - t|^{-1-\tau} \nonumber \\
	& \hspace*{1cm} + |s-y|^{-1} |t-z|^{-2} |s - t|^{-1-\tau} + |s-y|^{-1} |t-z|^{-1} |s - t|^{-2-\tau} \big] \dif s \dif t \nonumber \\
	& =: \mathbb I_1 + \mathbb I_2 + \mathbb I_3 + \mathbb I_4. \label{eq:bbI1234-MLSchroEqu2018}
	\end{align}
	{Denote $\mathbf D := \{x + x', x - x' \,;\, x, x' \in \widetilde{\mathcal D} \}$.
		Then we apply Lemma \ref{lem:intbdd-MLSchroEqu2018} to estimate $\mathbb I_j~(j = 1,2,3,4)$ as follows,}
	\begin{align}
	\mathbb I_1
	& = \iint_{{D_f} \times {D_f}} |s-y|^{-2} |t-z|^{-2} |s - t|^{-\tau} \dif s \dif t \nonumber \\
	& \leq \int_{\mathbf D} |s|^{-2} \big( \int_{\mathbf D} |t|^{-2} |t - (s+y-z)|^{-\tau} \dif t \big) \dif s \nonumber \\
	& \lesssim \int_{\mathbf D} |s|^{-2} [|s - (z-y)|^{3-2-\tau} + (\diam \mathbf D)^{3-2-\tau}] \dif s \nonumber \\
	& = C_{f} + \int_{\mathbf D} |s|^{-2} |s - (z-y)|^{-(\tau-1)} \dif s \nonumber \\
	& \lesssim C_{f} + |z-y|^{2-\tau} + (\diam \mathbf D)^{2-\tau} \nonumber \\
	& \simeq |z-y|^{2-\tau} + C_{f}.
	\label{eq:bbI1-MLSchroEqu2018}
	\end{align}
	Note that in \eqref{eq:bbI1-MLSchroEqu2018} we used Lemma \ref{lem:intbdd-MLSchroEqu2018} twice.
	Similarly,
	\begin{equation} \label{eq:bbI2-MLSchroEqu2018}
	\mathbb I_2,\, \mathbb I_3,\, \mathbb I_4
	\lesssim |z-y|^{2-\tau} + C_{f}.
	\end{equation}
	Recall that $\tau \in (0,1)$. By \eqref{eq:bbI1234-MLSchroEqu2018}, \eqref{eq:bbI1-MLSchroEqu2018} and \eqref{eq:bbI2-MLSchroEqu2018}, we arrive at
	\begin{align}
	|I(z,y)|
	& \leq C k_1^{-1} k_2^{-1} (|z-y|^{2-\tau} + C), \label{eq:I.inter3-MLSchroEqu2018}
	\end{align}
	where the constant $C$ is independent of $y$, $z$ and $k$.
	
	\smallskip

	Recall $V \in L_{3/2+\epsilon}^2(\R^3)$ stipulated in \eqref{eq:asp-MLSchroEqu2018}, so it follows $\nrm[L_{1/2+\epsilon}^2(\R^3)]{V} < +\infty$.
	This will be used in the next computation.
	Combining \eqref{eq:hotGjGl-MLSchroEqu2018} and \eqref{eq:I.inter3-MLSchroEqu2018} and (without loss of generality) assuming $\ell \geq 2$, one can compute
	\begin{align*}
		& \ |\mathbb{E} \big( G_j(k_1,\hat x_1) \cdot \overline{G_\ell(k_2,\hat x_2)} \,\big)| \\
		= & \ \big| \int e^{ik_2 \hat x_2 \cdot z} \bigg\{ (V \mathcal{R}_{k_2})^{\ell-1} \Big( \int e^{-ik_1 \hat x_1 \cdot y} \big[ (V \mathcal{R}_{k_1})^{j-1} (V(1) \overline V (2) I(2,1)) \big] (y) \dif{y} \Big) \bigg\}(z) \dif{z} \big| \\
		\leq & \ C_V \Nrm[L_{-1/2-\epsilon}^2(\R^3;2)]{\mathcal R_{k_2} (V \mathcal{R}_{k_2})^{\ell-2} \Big( \int e^{-ik_1 \hat x_1 \cdot y} \big[ (V \mathcal{R}_{k_1})^{j-1} (V(1) \overline V (2) I(2,1)) \big] (y) \dif{y} \Big)} \\
		\leq & \ C_{V} k_2^{-\ell+1} \Nrm[L_{1/2 + \epsilon}^2(\R^3;2)]{ \nrm[L_{-1/2-\epsilon}^2(\R^3;1)]{\mathcal R_{k_1} (V \mathcal{R}_{k_1})^{j-2} (V(1) \overline V (2) I(2,1))} } \\
		\lesssim & \ k_2^{-\ell+1} k_1^{-j+1} \Nrm[L_{1/2 + \epsilon}^2(\R^3;2)]{ \nrm[L_{1/2 + \epsilon}^2(\R^3;1)]{(V(1) \overline V (2) I(2,1))} } \\
	\end{align*}
	By substituting \eqref{eq:I.inter3-MLSchroEqu2018} into the computation above, we can continue
	\begin{align*}
		& \ |\mathbb{E} \big( G_j(k_1,\hat x_1) \cdot \overline{G_\ell(k_2,\hat x_2)} \,\big)| \\
		\lesssim & \ k_2^{-\ell+1} k_1^{-j+1} \Big( \iint \agl[y]^{1+2\epsilon} \agl[z]^{1+2\epsilon} |V(y) V(z) I(z,y)|^2 \dif y \dif z \Big)^{1/2} \\
		\leq & \ k_2^{-\ell} k_1^{-j} \Big( \iint \agl[y]^{1+2\epsilon} \agl[z]^{1+2\epsilon} |V(y) V(z)|^2 (|z-y|^{2-\tau} + C) \dif y \dif z \Big)^{1/2} \quad (\text{by~} \eqref{eq:I.inter3-MLSchroEqu2018}) \\
		\leq & \ k_2^{-\ell} k_1^{-j} \Big( \iint \agl[y]^{3+2\epsilon} \agl[z]^{3+2\epsilon} |V(y) V(z)|^2 \dif y \dif z + C \nrm[L_{1/2+\epsilon}^2(\R^3)]{V}^2 \Big)^{1/2} \\
		= & \ C k_2^{-\ell} k_1^{-j} \nrm[L_{3/2+\epsilon}^2(\R^3)]{V}
		< C k_2^{-\ell} k_1^{-j},
	\end{align*}
	where in the last inequality we used $\nrm[L_{3/2+\epsilon}^2(\R^3)]{V} < +\infty$ guaranteed by \eqref{eq:asp-MLSchroEqu2018}.
	We also used $|z-y| \leq \agl[z-y] \leq \agl[z] \agl[y]$.
	Therefore,
	\begin{align}
	& \Big|\sum_{j+\ell \geq 3,\,j,\ell \geq 1} \mathbb{E} \big( G_j(k_1,\hat x_1) \cdot \overline{G_\ell(k_2,\hat x_2)} \big) \Big| \nonumber\\
	\lesssim & \sum_{j = 1, \,\ell \geq 2} k_2^{-\ell} k_1^{-j} + \sum_{j \geq 2} \sum_{\ell \geq 1} k_2^{-\ell} k_1^{-j}
	\lesssim k_2^{-2} k_1^{-1} + \sum_{j \geq 2} \sum_{\ell \geq 1} k_2^{-2} k_1^{-j}
	\lesssim k^{-3}, \quad k \to +\infty. \label{eq:hotGjGlest-MLSchroEqu2018}
	\end{align}
	
	Finally, by combining \eqref{eq:F1F1G1GjGl-MLSchroEqu2018}, \eqref{eq:hotG1G1-MLSchroEqu2018} and \eqref{eq:hotGjGlest-MLSchroEqu2018}, we conclude \eqref{eq:hotF1F1-MLSchroEqu2018}, which completes the proof. 
\end{proof}

The following lemma is the ergodic version of Lemmas \ref{lemma:HOT0-MLSchroEqu2018} and \ref{lemma:HOT1-MLSchroEqu2018}.
\begin{lem} \label{lemma:HOTErgo-MLSchroEqu2018}
	Define $F_j(k,\hat x) ~(j=0,1)$ as in \eqref{eq:Fjkx-MLSchroEqu2018}. Write
	\begin{align*}
	X_{p,q}(K,\tau,\hat x) & = \frac 1 K \int_K^{2K} k^m \overline{F_q(k,\hat x)} \cdot F_p(k+\tau,\hat x) \dif{k}, \ \text{for} \ (p,q) \in \{ (0,1), (1,0), (1,1) \}.
	\end{align*}
	Then for any $\hat x \in \mathbb{S}^2$ and any $\tau \geq 0$, when $K \to +\infty$, we have the following estimates:
	{\small \begin{align} 
		\big| \mathbb{E} (X_{p,q}(K,\tau,\hat x)) \big| & = \mathcal{O}(K^{-1}), \quad \big| \mathbb{E} (|X_{p,q}(K,\tau,\hat x)|^2) \big| = \mathcal{O}(K^{-3/2}), \quad (p,q) \in \{ (0,1), (1,0) \}, \label{eq:hotF0F1Ergo-MLSchroEqu2018} \\
		\big| \mathbb{E} (X_{1,1}(K,\tau,\hat x)) \big| & = \mathcal{O}(K^{m-3}), \quad \big| \mathbb{E} (|X_{1,1}(K,\tau,\hat x)|^2) \big| = \mathcal{O}(K^{2(m-3)}). \label{eq:hotF1F1Ergo-MLSchroEqu2018}
		\end{align}}
	Let $\{K_j\} \in P \big( \max\{2/3,\, (3-m)^{-1}/2 \} + \gamma \big)$, then for any $\tau \geq 0$, we have
	\begin{equation} \label{eq:HOTErgoToZero-MLSchroEqu2018}
	\lim_{j \to +\infty} X_{p,q}(K_j,\tau,\hat x) = 0 \quad \ass,
	\end{equation}
	for every $(p,q) \in \{ (0,1), (1,0), (1,1) \}$.
\end{lem}

We may denote $X_{p,q}(K,\tau,\hat x)$ as $X_{p,q}$ for short if it is clear in the context.

\begin{proof}[Proof of Lemma \ref{lemma:HOTErgo-MLSchroEqu2018}] According to Lemmas \ref{lemma:HOT0-MLSchroEqu2018} and \ref{lemma:HOT1-MLSchroEqu2018}, we have
	\begin{align} 
	\mathbb{E} (X_{0,1}) & = \frac 1 K \int_K^{2K} k^m \mathbb{E} \big( \overline{F_1(k,\hat x)} \cdot F_0(k+\tau,\hat x) \big) \dif{k} = \frac 1 K \int_K^{2K} \mathcal{O}(k^{-1}) \dif{k} \nonumber\\
	& = \mathcal{O}(K^{-1}), \quad K \to +\infty. \label{eq:hotF0F1Ergo1-MLSchroEqu2018}
	\end{align}
	By formula \eqref{eq:I0EstE-MLSchroEqu2018}, Isserlis' Theorem and Lemma \ref{lemma:LeadingTermTechnical-MLSchroEqu2018}, we compute the secondary moment of $X_{0,1}$,
	{\small \begin{align}
		& \mathbb{E} \big( | X_{0,1} |^2 \big) \nonumber \\
		= & \ \mathbb{E} \Big( \frac 1 K \int_K^{2K} k_1^m F_0(k_1+\tau,\hat x) \cdot \overline{F_1(k_1,\hat x)} \dif{k_1} \cdot \frac 1 K \int_K^{2K} k_2^m \overline{ F_0(k_2+\tau,\hat x) } \cdot F_1(k_2,\hat x) \dif{k_2} \Big) \nonumber \\
		= & \ \frac 1 {K^2} \int_K^{2K} \int_K^{2K} [\mathcal{O}(K^{-2}) + (2\pi)^{3/2} \widehat{\mu} ((k_1-k_2) \hat x) \cdot \mathcal{O}(K^{-1}) + \mathcal{O}(K^{-2})] \dif{k_1} \dif{k_2} \nonumber\\
		= & \ \frac 1 {K^2} \int_K^{2K} \int_K^{2K} (2\pi)^{3/2} \widehat{\mu} ((k_1-k_2) \hat x) \dif{k_1} \dif{k_2} \cdot \mathcal{O}(K^{-1}) + \mathcal{O}(K^{-2}) \nonumber\\
		= & \ \mathcal{O}(K^{-1/2}) \cdot \mathcal{O}(K^{-1}) + \mathcal{O}(K^{-2}) \quad(\text{H\"older ineq. and } \eqref{eq:F0F0TermOne-MLSchroEqu2018}) \nonumber\\
		= & \ \mathcal{O}(K^{-3/2}), \quad K \to +\infty. \label{eq:hotF0F1Ergo2-MLSchroEqu2018}
		\end{align}}
	From \eqref{eq:hotF0F1Ergo1-MLSchroEqu2018} and \eqref{eq:hotF0F1Ergo2-MLSchroEqu2018} we obtain \eqref{eq:hotF0F1Ergo-MLSchroEqu2018} for $(p,q) = (0,1)$. Similarly, formula \eqref{eq:hotF0F1Ergo-MLSchroEqu2018} for $(p,q) = (1,0)$ can be proved and we skip the details.
	
	By Chebyshev's inequality and \eqref{eq:hotF0F1Ergo2-MLSchroEqu2018}, for any $\epsilon > 0$, we have
	\begin{align}
	& P \big( \bigcup_{j \geq K_0} \{ |X_{0,1}(K_j, \tau, \hat x) - 0| \geq \epsilon \} \big) \leq \frac C {\epsilon^2} \sum_{j \geq K_0} K_j^{-3/2} \leq \frac C {\epsilon^2} \sum_{j \geq K_0} j^{-1-3\gamma/2} \nonumber\\
	\leq & \frac C {\epsilon^2} \int_{K_0}^{+\infty} (t-1)^{-1-3\gamma/2} \dif{t} \to 0, \quad K_0 \to +\infty. \label{eq:X01Ergo-MLSchroEqu2018}
	\end{align}
	According to \cite[Lemma 3.3]{HLSM2019determining}, \eqref{eq:X01Ergo-MLSchroEqu2018} implies \eqref{eq:HOTErgoToZero-MLSchroEqu2018} for $(p,q) = (0,1)$. Similarly, formula \eqref{eq:HOTErgoToZero-MLSchroEqu2018} for $(p,q) = (1,0)$ can be proved.
	
	Now we prove \eqref{eq:hotF1F1Ergo-MLSchroEqu2018}. We have:
	{\small \begin{equation} \label{eq:hotF1F1Ergo1-MLSchroEqu2018}
		\mathbb{E} \big( X_{1,1} \big) = \frac 1 K \int_K^{2K} k^m \mathbb{E} \big( \overline{F_1(k,\hat x)} \cdot F_1(k+\tau,\hat x) \big) \dif{k}	= \frac 1 K \int_K^{2K} \mathcal{O}(K^{m-3}) \dif{k} = \mathcal{O}(K^{m-3}).
		\end{equation}}
	Compute the secondary moment:
	{\small \begin{align}
		& \mathbb{E} \big( | X_{1,1} |^2 \big) = \mathbb{E} \big( \frac 1 K \int_K^{2K} k_1^m F_1(k_1+\tau,\hat x) \cdot \overline{F_1(k_1,\hat x)} \dif{k_1} \cdot \frac 1 K \int_K^{2K} k_2^m \overline{ F_1(k_2+\tau,\hat x) } \cdot F_1(k_2,\hat x) \dif{k_2} \big) \nonumber\\
		= & \frac 1 {K^2} \int_K^{2K} \int_K^{2K} \mathcal{O}(K^{m-3}) \cdot \mathcal{O}(K^{m-3}) \dif{k_1} \dif{k_2} \quad (\text{Lemmas } \ref{lemma:HOT0-MLSchroEqu2018},\, \ref{lemma:HOT1-MLSchroEqu2018}) \nonumber\\
		= & \mathcal{O}(K^{2(m-3)}), \quad K \to +\infty. \label{eq:hotF1F1Ergo2-MLSchroEqu2018}
		\end{align}}
	Formulae \eqref{eq:hotF1F1Ergo1-MLSchroEqu2018} and \eqref{eq:hotF1F1Ergo2-MLSchroEqu2018} gives \eqref{eq:hotF1F1Ergo-MLSchroEqu2018}.
	
	By Chebyshev's inequality and \eqref{eq:hotF1F1Ergo2-MLSchroEqu2018}, for any $\epsilon > 0$, we have
	\begin{align}
	& P \big( \bigcup_{j \geq K_0} \{ |X_{1,1} - 0| \geq \epsilon \} \big) \leq \frac C {\epsilon^2} \sum_{j \geq K_0} K_j^{2(m-3)} \leq \frac C {\epsilon^2} \sum_{j \geq K_0} j^{-1-\gamma'} \nonumber\\
	\leq & \frac C {\epsilon^2} \int_{K_0}^{+\infty} (t-1)^{-1-\gamma'} \dif{t} \to 0, \quad K_0 \to +\infty., \label{eq:X11Ergo-MLSchroEqu2018}
	\end{align}
	where $\gamma'$ is some positive constant depending on $m$.
	According to \cite[Lemma 3.3]{HLSM2019determining}, \eqref{eq:X11Ergo-MLSchroEqu2018} implies \eqref{eq:HOTErgoToZero-MLSchroEqu2018} for $(p,q) = (1,1)$.
	The proof is complete.
\end{proof}

\section{The recovery of the rough strength} \label{sec:VarRec-MLSchroEqu2018}

In this section we focus on the recovery of the rough strength $\mu(x)$ of the random source. We employ only a single-realisation of the passive scattering measurement, namely the random sample $\omega$ is fixed. 
The data set $\{ u^\infty(\hat x, k, \omega) \,\big|\, \hat x \in \mathbb{S}^2, k \in \R_+ \}$ is utilized to achieve the unique recovery result. 
In what follows, we present the main results of recovering $\mu(x)$ in Section \ref{subsec:MainSteps-MLSchroEqu2018}, and put the corresponding proofs in Section \ref{subsec:ProofsToMainSteps-MLSchroEqu2018}. The auxiliary lemmas derived in Section \ref{subsec:AsympHighOrder-MLSchroEqu2018} shall play a key role to the proofs in Section \ref{subsec:ProofsToMainSteps-MLSchroEqu2018}.

\subsection{Main unique recovery results} \label{subsec:MainSteps-MLSchroEqu2018}

The first main recovery result is given as follows.

\begin{thm} \label{lem:sigmaHatRec-MLSchroEqu2018}
	We have the following asymptotic identity,
	\begin{equation} \label{eq:sigmaHatRec-MLSchroEqu2018}
	4\sqrt{2\pi} \lim_{k \to +\infty} \mathbb{E} \big( k^m \big[\, \overline{ u^\infty(\hat x,k) } -  \overline{ \mathbb{E} u^\infty(\hat x,k) }\, \big] \cdot \big[ u^\infty(\hat x,k+\tau) - \mathbb{E} u^\infty(\hat x,k+\tau) \big] \big) = \widehat{\mu}(\tau \hat x),
	\end{equation}
	where $\tau \geq 0,~ \hat x \in \mathbb{S}^2$, and $\widehat u$ is the Fourier transform of $\mu$.
\end{thm}

Theorem \ref{lem:sigmaHatRec-MLSchroEqu2018} clearly yields a recovery formula for the rough strength $\mu$. However, it requires many realizations and is lack of practical usefulness. The result in Theorem \ref{lem:sigmaHatRec-MLSchroEqu2018} can be improved by using the ergodicity as follows.


\begin{thm} \label{lem:sigmaHatRecErgo-MLSchroEqu2018}
	Let $m^* = \max \{2/3,\, (3-m)^{-1}/2 \}$. 
	Assume that $\{K_j\} \in P(m^* + \gamma)$. 
	Then  $\exists\, \Omega_0 \subset \Omega \colon \mathbb{P}(\Omega_0) = 0$, $\Omega_0$ depending only on $\{K_j\}_{j \in \mathbb{N}}$, such that for any $\omega \in \Omega \backslash \Omega_0$, there exists $S_\omega \subset \R^3 \colon |S_\omega| = 0$, it holds that for $\forall \tau \in \R_+$ and $\forall \hat x \in \mathbb{S}^2$ satisfying $\tau \hat x \in \R^3 \backslash S_\omega$,
	{\small \begin{align} 
		& 4\sqrt{2\pi} \lim_{j \to +\infty} \frac 1 {K_j} \int_{K_j}^{2K_j} k^m 
		\big[\, \overline{u^\infty(\hat x,k,\omega)} - \overline{\mathbb{E} u^\infty(\hat x,k)} \,\big]
		\cdot
		\big[ u^\infty(\hat x,k+\tau,\omega) - \mathbb{E} u^\infty(\hat x,k+\tau) \big] 
		\dif{k} \nonumber\\
		& = \widehat{\mu} (\tau \hat x). \label{eq:SecondOrderErgo-MLSchroEqu2018}
		\end{align}}
\end{thm}

The recovery formula presented in \eqref{eq:SecondOrderErgo-MLSchroEqu2018} still involves all the realizations of the random sample $\omega$ due to the presence of the term $\mathbb E (u^\infty\hat x,k))$. To recover $\mu(x)$ by only one realization of the passive scattering measurement, the $\mathbb{E}(u^\infty(\hat x,k))$  should be further relaxed in \eqref{eq:SecondOrderErgo-MLSchroEqu2018}, and this is done by Theorem \ref{lem:sigmaHatRecSingle-MLSchroEqu2018} in the following.

\begin{thm} \label{lem:sigmaHatRecSingle-MLSchroEqu2018}
	Under the same condition as in Theorem \ref{lem:sigmaHatRecErgo-MLSchroEqu2018}, we have
	\begin{equation} \label{eq:sigmaHatRecSingle-MLSchroEqu2018}
	4\sqrt{2\pi} \lim_{j \to +\infty} \frac 1 {K_j} \int_{K_j}^{2K_j} k^m \overline{u^\infty(\hat x,k,\omega)} \cdot u^\infty(\hat x,k+\tau,\omega) \dif{k} = \widehat{\mu} (\tau \hat x),
	\end{equation}
	holds for $\forall \tau \in \R_+$ and $\forall \hat x \in \mathbb{S}^2$ satisfying $\tau \hat x \in \R^3 \backslash S_\omega$.
\end{thm}

Now Theorem \ref{thm:Unisigma-MLSchroEqu2018} becomes a direct consequence of Theorem \ref{lem:sigmaHatRecSingle-MLSchroEqu2018}. 

\begin{proof}[Proof of Theorem \ref{thm:Unisigma-MLSchroEqu2018}]
	Theorem \ref{lem:sigmaHatRecSingle-MLSchroEqu2018} provides a recovery formula for the local strength $\mu$ by the far-field data $\{ u^{\infty}(\hat{x},k,\omega); \Forall \hat{x} \in \mathbb{S}^2, \Forall k \in \R_+ \}$ with a single fixed $\omega\in\Omega$.
\end{proof}


\subsection{Proofs of the main theorems} \label{subsec:ProofsToMainSteps-MLSchroEqu2018}

In this subsection, we present the proofs of Theorems \ref{lem:sigmaHatRec-MLSchroEqu2018}, \ref{lem:sigmaHatRecErgo-MLSchroEqu2018} and \ref{lem:sigmaHatRecSingle-MLSchroEqu2018}.

\begin{proof}[Proof of Theorem \ref{lem:sigmaHatRec-MLSchroEqu2018}]
	Let $k$ be large enough s.t.~$(I - \Rk V)^{-1} = \sum_{j=0}^{+\infty} (\Rk V)^j$, and let $\tau \in \R_+$. According to the analysis at the beginning of Section \ref{sec:AsyEst-MLSchroEqu2018}, one can compute
	\begin{align}
	& 16\pi^2  \mathbb{E} \big( [ \overline{u^\infty(\hat x,k) - \mathbb{E}u^\infty(\hat x,k)} ] [ u^\infty(\hat x,k+\tau) - u^\infty(\hat x,k)] \big) \nonumber\\
	= & \ \sum_{j,\ell = 0,1} \mathbb{E} \big( \overline{F_\ell(k,\hat x)}  F_j(k+\tau,\hat x) \big) =: I_{0,0} + I_{0,1} + I_{1,0} + I_{1,1}. \label{eq:Thm1uI-MLSchroEqu2018}
	\end{align}
	From Lemmas \ref{lemma:HOT0-MLSchroEqu2018} and \ref{lemma:HOT1-MLSchroEqu2018}, we have that
	$I_{0,1}$, $I_{1,0}$, $I_{1,1}$ are all of the order no less than $k^{-3}$, and hence
	\begin{equation} \label{eq:u1Infty-MLSchroEqu2018}
	16\pi^2  \mathbb{E} \big( [ \overline{u^\infty(\hat x,k) - \mathbb{E}u^\infty(\hat x,k)} ] [ u^\infty(\hat x,k+\tau) - u^\infty(\hat x,k)] \big) = k^m I_{0,0} + \mathcal{O}(k^{m-3}),
	\end{equation}
	as $k$ goes to infinity.
	Then, \eqref{eq:I0-MLSchroEqu2018} gives
	\begin{align*}
	I_{0,0} & = \mathbb{E} \big( \overline{F_0(k,\hat x)} F_0(k+\tau,\hat x) \big)
	= (2\pi)^{3/2}\, \widehat{\mu} (\tau \hat x) k^{-m} + \int_{\mathcal D} a(y,k \hat x) e^{i\tau \hat x \cdot y} \dif{y}.
	\end{align*}
	The symbol $a$ is of order $-m-1$, and thus
	\begin{equation} \label{eq:aEst-MLSchroEqu2018}
		\big| \int_{\mathcal D} a(y,k \hat x) e^{i\tau \hat x \cdot y} \dif{y} \big| \leq |\mathcal D| \cdot |a(y,k \hat x)| \leq |\mathcal D| C \agl[k \hat x]^{-m-1} = |\mathcal D| C \agl[k]^{-m-1}.
	\end{equation}
	From \eqref{eq:aEst-MLSchroEqu2018} we obtain
	\begin{equation} \label{eq:I0Est-MLSchroEqu2018}
	k^m I_{0,0} = \mathbb{E} \big( k^m \overline{F_0(k,\hat x)} F_0(k+\tau,\hat x) \big) = (2\pi)^{3/2}\, \widehat{\mu} (\tau \hat x) + \mathcal{O}(k^{-1}), \quad k \to +\infty.
	\end{equation}
	Formulae \eqref{eq:u1Infty-MLSchroEqu2018} and \eqref{eq:I0Est-MLSchroEqu2018} give 
	{\small\begin{equation} \label{eq:sigmaHatRecZeroMeanInter-MLSchroEqu2018}
		4 \sqrt{2\pi} \mathbb{E} \big( [ \overline{u^\infty(\hat x,k) - \mathbb{E}u^\infty(\hat x,k)} ] [ u^\infty(\hat x,k+\tau) - u^\infty(\hat x,k)] \big) = \widehat{\mu}(\tau \hat x) + \mathcal{O}(k^{m-3}) + \mathcal{O}(k^{-1}),
		\end{equation}}
	as $k$ goes to infinity.
	Noting that $m \in (1,3)$, \eqref{eq:sigmaHatRecZeroMeanInter-MLSchroEqu2018} immediately implies \eqref{eq:sigmaHatRec-MLSchroEqu2018}.
\end{proof}

\begin{proof}[Proof of Theorem \ref{lem:sigmaHatRecErgo-MLSchroEqu2018}]
	For convenience, we denote the averaging operation with respect to $k$ as $\mathcal{E}_k$, i.e.~${{\mathcal{E}_k f} = \frac 1 K \int_K^{2K} f(k) \dif{k}}$. 
	Similar to \eqref{eq:Thm1uI-MLSchroEqu2018}, we have
	\begin{align}
	& 16\pi^2 \mathcal{E}_k \big( k^m [ \overline{u^\infty(\hat x,k) - \mathbb{E}u^\infty(\hat x,k+\tau)} ] [ u^\infty(\hat x,k+\tau) - \mathbb{E} u^\infty(\hat x,k+\tau) ] \big) \nonumber\\
	= & \sum_{j,\ell = 0,1} \mathcal{E}_k \big( k^m \overline{F_\ell(k,\hat x)} F_j(k+\tau,\hat x) \big) =: X_{0,0} + X_{0,1}+ X_{1,0} + X_{1,1}. \label{eq:Thm2uX-MLSchroEqu2018}
	\end{align}
	Recall that $\{K_j\} \in P(m^*+\gamma)$.
	For $\forall \tau \geq 0$ and $\forall \hat x \in \mathbb{S}^2$, Lemma \ref{lemma:LeadingTermErgo-MLSchroEqu2018} implies that $\exists \Omega_{\tau,\hat x}^{0,0} \subset \Omega \colon \mathbb{P}(\Omega_{\tau,\hat x}^{0,0}) = 0$, $\Omega_{\tau,\hat x}^{0,0}$ depending on $\tau$ and $\hat x$, such that
	\begin{equation} \label{eq:Thm2X00-MLSchroEqu2018}
	\lim_{j \to +\infty} X_{0,0}(K_j,\tau,\hat x) = (2\pi)^{3/2} \widehat{\mu} (\tau \hat x), \quad \forall \omega \in \Omega \backslash \Omega_{\tau,\hat x}^{0,0}.
	\end{equation}
	Lemma \ref{lemma:HOTErgo-MLSchroEqu2018} implies the existence of the sets $\Omega_{\tau,\hat x}^{p,q} ~\big( (p,q) \in \{ (0,1),\, (1,0),\, (1,1) \} \big)$ with zero probability measures such that $\forall \tau \geq 0$ and $\forall \hat x \in \mathbb{S}^2$,
	\begin{equation} \label{eq:Thm2Xpq-MLSchroEqu2018}
	\lim_{j \to +\infty} X_{p,q}(K_j,\tau,\hat x) = 0, \quad \forall \omega \in \Omega \backslash \Omega_{\tau,\hat x}^{p,q}.
	\end{equation}
	for all $(p,q) \in \{ (0,1),\, (1,0),\, (1,1) \}$. Write $\Omega_{\tau,\hat x} = \bigcup_{p,q = 0,1} \Omega_{\tau,\hat x}^{p,q}$\,, then $\mathbb{P} (\Omega_{\tau,\hat x}) = 0$. 
	From Lemmas \ref{lemma:LeadingTermErgo-MLSchroEqu2018} and \ref{lemma:HOTErgo-MLSchroEqu2018} we note that $\Omega_{\tau,\hat x}^{p,q}$ also depends on $K_j$, so does $\Omega_{\tau,\hat x}$, but we omit this dependence in the notation. Write
	$$Z(\tau\hat x,\omega) := \lim_{j \to +\infty}  \frac {16\pi^2} {K_j} \int_{K_j}^{2K_j} k^m \overline{u^\infty(\hat x,k)} u^\infty(\hat x,k+\tau) \dif{k} - (2\pi)^{3/2} \widehat{\mu} (\tau \hat x)$$
	for short. 
	By \eqref{eq:Thm2uX-MLSchroEqu2018}, \eqref{eq:Thm2X00-MLSchroEqu2018} and \eqref{eq:Thm2Xpq-MLSchroEqu2018}, we conclude that
	\begin{equation} \label{eq:SecondOrderErgo2-MLSchroEqu2018}
	\Forall y \in \R^3, \Exists \Omega_y \subset \Omega \colon \mathbb P (\Omega_y) = 0, \st \forall\, \omega \in \Omega \backslash \Omega_y,\, Z(y,\omega) = 0.
	\end{equation}
	
	To conclude \eqref{eq:SecondOrderErgo-MLSchroEqu2018} from \eqref{eq:SecondOrderErgo2-MLSchroEqu2018}, we need to exchange the order between $y$ and $\omega$.
	To achieve this, we utilize the Fubini's Theorem.
	Denote the usual Lebesgue measure on $\R^3$ as $\mathbb L$ and the product measure $\mathbb L \times \mathbb P$ as $\mu$, and construct the product measure space $\mathbb M := (\R^3 \times \Omega, \mathcal G, \mu)$ in the canonical way, where $\mathcal G$ is the corresponding complete $\sigma$-algebra.
	Write
	\[
	\mathcal{A} := \{ (y,\omega) \in \R^3 \times \Omega \,;\, Z(y, \omega) \neq 0 \}.
	\]
	Set $\chi_\mathcal{A}$ as the characteristic function of $\mathcal{A}$ in $\mathbb M$. By \eqref{eq:SecondOrderErgo2-MLSchroEqu2018} we obtain
	\begin{equation} \label{eq:FubiniEq0-MLSchroEqu2018}
	\int_{\R^3} \big( \int_\Omega \chi_{\mathcal A}(y,\omega) \dif{\mathbb P(\omega)} \big) \dif{\mathbb L(y)} = 0.
	\end{equation}
	By \eqref{eq:FubiniEq0-MLSchroEqu2018} and \cite[Corollary 7 in Section 20.1]{royden2000real}, we obtain
	\begin{equation} \label{eq:FubiniEq1-MLSchroEqu2018}
	\int_{\mathbb M} \chi_{\mathcal A}(y,\omega) \dif{\mathbb \mu} = \int_\Omega \big( \int_{\R^3} \chi_{\mathcal A}(y,\omega) \dif{\mathbb L(y)} \big) \dif{\mathbb P(\omega)} = 0.
	\end{equation}
	Because $\chi_{\mathcal A}(y,\omega)$ is nonnegative, \eqref{eq:FubiniEq1-MLSchroEqu2018} implies
	\begin{equation} \label{eq:FubiniEq2-MLSchroEqu2018}
	\Exists \Omega_0 \colon \mathbb P (\Omega_0) = 0, \st \forall\, \omega \in \Omega \backslash \Omega_0,\, \int_{\R^3} \chi_{\mathcal A}(y,\omega) \dif{\mathbb L(y)} = 0.
	\end{equation}
	Formula \eqref{eq:FubiniEq2-MLSchroEqu2018} further implies for every $\omega \in \Omega \backslash \Omega_0$,
	\begin{equation} \label{eq:FubiniEq3-MLSchroEqu2018}
	\Exists S_\omega \subset \R^3 \colon \mathbb L (S_\omega) = 0, \st \forall\, y \in \R^3 \backslash S_\omega,\, Z(y,\omega) = 0.
	\end{equation}
	This is \eqref{eq:SecondOrderErgo-MLSchroEqu2018}.
	The proof is complete.
\end{proof}

\begin{proof}[Proof of Theorem \ref{lem:sigmaHatRecSingle-MLSchroEqu2018}]
	Let $\mathcal{E}_k$ be the averaging operator as defined in the proof of Theorem \ref{lem:sigmaHatRecErgo-MLSchroEqu2018}. 
	For convenience, we denote $u_0^\infty(\hat x,k) = u^\infty(\hat x,k) - \mathbb{E}u^\infty(\hat x,k)$, we write $u_1^\infty(\hat x,k) = \mathbb{E}u^\infty(\hat x,k)$, thus $u^\infty = u_0^\infty + u_1^\infty$.
	And we have
	\begin{align*}
	& \ 16\pi^2 \mathcal{E}_k \big( k^m \overline{u^\infty(\hat x,k)} u^\infty(\hat x,k+\tau) \big) = 16\pi^2 \sum_{p,q = 0,1} \mathcal{E}_k \big( k^m \overline{u_p^\infty(\hat x,k)} u_q^\infty(\hat x,k+\tau) \big) \\
	=: & \ J_{0,0} + J_{0,1} + J_{1,0} + J_{1,1}.
	\end{align*}
	From Theorem \ref{lem:sigmaHatRecErgo-MLSchroEqu2018} we obtain
	\begin{equation} \label{eq:J0-MLSchroEqu2018}
	\begin{split}
	& \lim_{j \to +\infty} J_{0,0} = 16\pi^2 \lim_{j \to +\infty} \int_{K_j}^{2K_j} k^m \overline{u_0^\infty(\hat x,k)} \cdot u_0^\infty(\hat x,k+\tau) \dif{k} = (2\pi)^{3/2} \widehat{\mu} (\tau \hat x), \\
	& \quad \tau \hat x \aee \! \in \R^3, \quad \omega \ass \! \in \Omega.
	\end{split}
	\end{equation}
	
	Then we study $J_{0,1}$,
	\begin{align}
	|J_{0,1}|^2
	& \simeq \big| \mathcal{E}_k \big( k^m \overline{u_0^\infty(\hat x,k)} u_1^\infty(\hat x,k+\tau) \big) \big|^2
	= \big| \frac 1 {K_j} \int_{K_j}^{2K_j} k^m \overline{u_0^\infty(\hat x,k)} u_1^\infty(\hat x,k+\tau) \dif{k} \big|^2 \nonumber\\
	& \leq \frac 1 {K_j} \int_{K_j}^{2K_j} k^m |u_0^\infty(\hat x,k)|^2 \dif{k} \cdot \frac 1 {K_j} \int_{K_j}^{2K_j} k^m |u_1^\infty(\hat x,k+\tau)|^2 \dif{k}. \label{eq:J1One-MLSchroEqu2018}
	\end{align}
	Combining \eqref{eq:J1One-MLSchroEqu2018} with Theorem \ref{lem:sigmaHatRecErgo-MLSchroEqu2018} and Lemma \ref{lemma:FarFieldGoToZero-MLSchroEqu2018}, we obtain
	\begin{equation} \label{eq:J1-MLSchroEqu2018}
	|J_{1,2}|^2 \lesssim (\widehat{\sigma^2}(0) + o(1)) \cdot \mathcal O(k^{m-4}) = o(1) \to 0, \quad j \to +\infty.
	\end{equation}
	The analysis to $J_{1,0}$ is similar to that of $J_{0,1}$, so we skip the details.
	
	Finally we analyze $J_{1,1}$. By Lemma \ref{lemma:FarFieldGoToZero-MLSchroEqu2018}, we have
	\begin{align}
	|J_{1,1}|^2
	& \simeq \big| \mathcal{E}_k \big( k^m \overline{u_1^\infty(\hat x,k)} u_1^\infty(\hat x,k+\tau) \big) \big|^2
	= \big| \frac 1 {K_j} \int_{K_j}^{2K_j} k^m  \overline{u_1^\infty(\hat x,k)} u_1^\infty(\hat x,k+\tau) \dif{k} \big|^2 \nonumber\\
	& \leq \frac 1 {K_j} \int_{K_j}^{2K_j} k^m |u_1^\infty(\hat x,k)|^2 \dif{k} \cdot \frac 1 {K_j} \int_{K_j}^{2K_j} k^m |u_1^\infty(\hat x,k+\tau)|^2 \dif{k} \nonumber\\
	& \leq \frac 1 {K_j} \int_{K_j}^{2K_j} k^m \sup_{\kappa \geq K_j} \big| u_1^\infty(\hat x,\kappa) \big|^2 \dif{k} \cdot \frac 1 {K_j} \int_{K_j}^{2K_j} k^m \sup_{\kappa \geq K_j+\tau} \big| u_1^\infty(\hat x,\kappa) \big|^2 \dif{k} \nonumber\\
	& = (2 K_j)^m \sup_{\kappa \geq K_j} |u_1^\infty(\hat x,\kappa)|^2 \cdot \sup_{\kappa \geq K_j+\tau} |u_1^\infty(\hat x,\kappa)|^2 \to 0, \quad j \to +\infty. \label{eq:J3-MLSchroEqu2018}
	\end{align}
	Combining \eqref{eq:J0-MLSchroEqu2018}, \eqref{eq:J1-MLSchroEqu2018} and \eqref{eq:J3-MLSchroEqu2018}, we can conclude \eqref{eq:sigmaHatRecSingle-MLSchroEqu2018}.
	The proof is complete. 
\end{proof}

%
%

\section*{Acknowledgements}
The work of H.~Liu was supported by Hong Kong RGC general research funds (projects 11311122, 12301420 and 11300821).

\section*{Statements and Declarations}

There are no competing interests of the authors to disclose. The authors contribute equally to this piece of theoretical work.




\begin{bibdiv}
	\begin{biblist}
		
		\bib{bao2016inverse}{article}{
			author={Bao, Gang},
			author={Chen, Chuchu},
			author={Li, Peijun},
			title={Inverse random source scattering problems in several dimensions},
			date={2016},
			journal={SIAM/ASA J. Uncertain. Quantif.},
			volume={4},
			number={1},
			pages={1263\ndash 1287},
		}
		
		\bib{Bsource}{article}{
			author={Bl{\aa}sten, Eemeli},
			title={Nonradiating sources and transmission eigenfunctions vanish at
				corners and edges},
			date={2018},
			journal={SIAM J. Math. Anal.},
			volume={50},
			number={6},
			pages={6255\ndash 6270},
		}
		
		\bib{BL2018}{article}{
			author={Bl{\aa}sten, Emilia},
			author={Liu, Hongyu},
			title={Scattering by curvatures, radiationless sources, transmission
				eigenfunctions, and inverse scattering problems},
			date={2021},
			ISSN={0036-1410},
			journal={SIAM J. Math. Anal.},
			volume={53},
			number={4},
			pages={3801\ndash 3837},
			url={https://doi.org/10.1137/20M1384002},
			review={\MR{4283699}},
		}
		
		\bib{blomgren2002}{article}{
			author={Blomgren, Peter},
			author={Papanicolaou, George},
			author={Zhao, Hongkai},
			title={Super-resolution in time-reversal acoustics},
			date={2002},
			journal={The Journal of the Acoustical Society of America},
			volume={111},
			number={1},
			pages={230\ndash 248},
		}
		
		\bib{borcea2006}{article}{
			author={Borcea, Liliana},
			author={Papanicolaou, George},
			author={Tsogka, Chrysoula},
			title={Adaptive interferometric imaging in clutter and optimal
				illumination},
			date={2006},
			journal={Inverse Probl.},
			volume={22},
			number={4},
			pages={1405},
		}
		
		\bib{borcea2002}{article}{
			author={Borcea, Liliana},
			author={Papanicolaou, George},
			author={Tsogka, Chrysoula},
			author={Berryman, James},
			title={Imaging and time reversal in random media},
			date={2002},
			journal={Inverse Probl.},
			volume={18},
			number={5},
			pages={1247},
		}
		
		\bib{caro2016inverse}{article}{
			author={Caro, Pedro},
			author={Helin, Tapio},
			author={Lassas, Matti},
			title={Inverse scattering for a random potential},
			date={2019},
			journal={Anal. Appl.},
			volume={17},
			number={4},
			pages={513\ndash 567},
		}
		
		\bib{ClaKli}{article}{
			author={Clason, Christian},
			author={Klibanov, Michael~V},
			title={The quasi-reversibility method for thermoacoustic tomography in a
				heterogeneous medium},
			date={2007},
			journal={SIAM J. Sci. Comput.},
			volume={30},
			number={1},
			pages={1\ndash 23},
		}
		
		\bib{colton2012inverse}{book}{
			author={Colton, David},
			author={Kress, Rainer},
			title={Inverse acoustic and electromagnetic scattering theory},
			series={Applied Mathematical Sciences},
			publisher={Springer, Cham},
			date={[2019] \copyright 2019},
			volume={93},
			ISBN={978-3-030-30350-1; 978-3-030-30351-8},
			url={https://doi.org/10.1007/978-3-030-30351-8},
			note={Fourth edition of [ MR1183732]},
			review={\MR{3971246}},
		}
		
		\bib{Deng2019onident}{article}{
			author={Deng, Youjun},
			author={Li, Jinhong},
			author={Liu, Hongyu},
			title={On identifying magnetized anomalies using geomagnetic
				monitoring},
			date={2019},
			ISSN={0003-9527},
			journal={Arch. Ration. Mech. Anal.},
			volume={231},
			number={1},
			pages={153\ndash 187},
			url={https://doi.org/10.1007/s00205-018-1276-7},
		}
		
		\bib{Deng2020onident}{article}{
			author={Deng, Youjun},
			author={Li, Jinhong},
			author={Liu, Hongyu},
			title={On identifying magnetized anomalies using geomagnetic monitoring
				within a magnetohydrodynamic model},
			date={2020},
			ISSN={0003-9527},
			journal={Arch. Ration. Mech. Anal.},
			volume={235},
			number={1},
			pages={691\ndash 721},
			url={https://doi.org/10.1007/s00205-019-01429-x},
		}
		
		\bib{Deng2019Onan}{article}{
			author={Deng, Youjun},
			author={Liu, Hongyu},
			author={Uhlmann, Gunther},
			title={On an inverse boundary problem arising in brain imaging},
			date={2019},
			ISSN={0022-0396},
			journal={J. Differential Equations},
			volume={267},
			number={4},
			pages={2471\ndash 2502},
			url={https://doi.org/10.1016/j.jde.2019.03.019},
		}
		
		\bib{eskin2011lectures}{book}{
			author={Eskin, Gregory},
			title={Lectures on linear partial differential equations},
			series={Graduate Studies in Mathematics},
			publisher={American Mathematical Society, Providence, RI},
			date={2011},
			volume={123},
			ISBN={978-0-8218-5284-2},
			url={https://doi.org/10.1090/gsm/123},
			review={\MR{2809923}},
		}
		
		\bib{griffiths2016introduction}{book}{
			author={Griffiths, D.~J.},
			title={Introduction to quantum mechanics},
			publisher={Cambridge Univ. Press},
			address={Cambridge},
			date={2016},
		}
		
		\bib{hormander1985analysisI}{book}{
			author={H\"{o}rmander, Lars},
			title={The analysis of linear partial differential operators. {I}},
			series={Classics in Mathematics},
			publisher={Springer-Verlag, Berlin},
			date={2003},
			ISBN={3-540-00662-1},
			url={https://doi.org/10.1007/978-3-642-61497-2},
			note={Distribution theory and Fourier analysis, Reprint of the second
				(1990) edition [Springer, Berlin; MR1065993 (91m:35001a)]},
			review={\MR{1996773}},
		}
		
		\bib{hormander1985analysisIII}{book}{
			author={H\"{o}rmander, Lars},
			title={The analysis of linear partial differential operators. {III}},
			series={Classics in Mathematics},
			publisher={Springer, Berlin},
			date={2007},
			ISBN={978-3-540-49937-4},
			url={https://doi.org/10.1007/978-3-540-49938-1},
			note={Pseudo-differential operators, Reprint of the 1994 edition},
			review={\MR{2304165}},
		}
		
		\bib{Klibanov2013}{article}{
			author={Klibanov, M.},
			title={Thermoacoustic tomography with an arbitrary elliptic operator},
			date={2013},
			journal={Inverse Probl.},
			volume={29},
			pages={025014},
		}
		
		\bib{KM1}{article}{
			author={Knox, Christina},
			author={Moradifam, Amir},
			title={Determining both the source of a wave and its speed in a medium
				from boundary measurements},
			date={2020},
			ISSN={0266-5611},
			journal={Inverse Probl.},
			volume={36},
			number={2},
			pages={025002, 15},
			url={https://doi.org/10.1088/1361-6420/ab53fc},
			review={\MR{4063195}},
		}
		
		\bib{LassasA}{article}{
			author={Lassas, M.},
			author={P{\"{a}}iv{\"{a}}rinta, L.},
			author={Saksman, E.},
			title={Inverse problem for a random potential},
			date={2004},
			journal={Contemp. Math. Amer. Math. Soc., Providence, RI.},
			volume={362},
		}
		
		\bib{Lassas2008}{article}{
			author={Lassas, M.},
			author={P{\"{a}}iv{\"{a}}rinta, L.},
			author={Saksman, E.},
			title={Inverse scattering problem for a two dimensional random
				potential},
			date={2008},
			journal={Comm. Math. Phys.},
			volume={279},
			pages={669\ndash 703},
		}
		
		\bib{LiHelinLiinverse2018}{article}{
			author={Li, Jianliang},
			author={Helin, Tapio},
			author={Li, Peijun},
			title={Inverse random source problems for time-harmonic acoustic and
				elastic waves},
			date={2020},
			ISSN={0360-5302},
			journal={Comm. Partial Differential Equations},
			volume={45},
			number={10},
			pages={1335\ndash 1380},
			url={https://doi.org/10.1080/03605302.2020.1774895},
			review={\MR{4160439}},
		}
		
		\bib{LiLiinverse2018}{article}{
			author={Li, Jianliang},
			author={Li, Peijun},
			title={Inverse elastic scattering for a random source},
			date={2019},
			ISSN={0036-1410},
			journal={SIAM J. Math. Anal.},
			volume={51},
			number={6},
			pages={4570\ndash 4603},
			url={https://doi.org/10.1137/18M1235119},
			review={\MR{4029816}},
		}
		
		\bib{HLSM2019determining}{article}{
			author={Li, Jingzhi},
			author={Liu, Hongyu},
			author={Ma, Shiqi},
			title={Determining a random {S}chr\"{o}dinger equation with unknown
				source and potential},
			date={2019},
			ISSN={0036-1410},
			journal={SIAM J. Math. Anal.},
			volume={51},
			number={4},
			pages={3465\ndash 3491},
			url={https://doi.org/10.1137/18M1225276},
			review={\MR{3995040}},
		}
		
		\bib{HLSM2019both}{article}{
			author={Li, Jingzhi},
			author={Liu, Hongyu},
			author={Ma, Shiqi},
			title={Determining a {R}andom {S}chr\"{o}dinger {O}perator: {B}oth
				{P}otential and {S}ource are {R}andom},
			date={2021},
			ISSN={0010-3616},
			journal={Comm. Math. Phys.},
			volume={381},
			number={2},
			pages={527\ndash 556},
			url={https://doi.org/10.1007/s00220-020-03889-9},
			review={\MR{4207450}},
		}
		
		\bib{liu2015determining}{article}{
			author={Liu, Hongyu},
			author={Uhlmann, Gunther},
			title={Determining both sound speed and internal source in thermo-and
				photo-acoustic tomography},
			date={2015},
			journal={Inverse Probl.},
			volume={31},
			number={10},
			pages={105005},
		}
		
		\bib{Lu1}{article}{
			author={L\"u, Q.},
			author={Zhang, X},
			title={Global uniqueness for an inverse stochastic hyperbolic problem
				with three unknowns},
			date={2015},
			journal={Comm. Pure Appl. Math.},
			volume={68},
			pages={948\ndash 963},
		}
		
		\bib{SM2020onrecent}{article}{
			author={Ma, Shiqi},
			title={On recent progress of single-realization recoveries of random
				{S}chr\"{o}dinger systems},
			date={2021},
			journal={Electron. Res. Arch.},
			volume={29},
			number={3},
			pages={2391\ndash 2415},
			url={https://doi.org/10.3934/era.2020121},
			review={\MR{4256480}},
		}
		
		\bib{royden2000real}{book}{
			author={Royden, H.~L.},
			author={Fitzpatrick, P.~M.},
			title={Real analysis},
			publisher={Prentice Hall},
			date={2010},
			ISBN={9780131437470},
		}
		
		\bib{uhlmann2013inverse}{book}{
			editor={Uhlmann, Gunther},
			title={Inverse problems and applications: inside out. {II}},
			series={Mathematical Sciences Research Institute Publications},
			publisher={Cambridge University Press, Cambridge},
			date={2013},
			volume={60},
			ISBN={978-1-107-03201-9},
		}
		
		\bib{WangGuo17}{article}{
			author={Wang, X.},
			author={Guo, Y.},
			author={Zhang, D.},
			author={Liu, H.},
			title={Fourier method for recovering acoustic sources from
				multi-frequency far-field data},
			date={2017},
			journal={Inverse Probl.},
			volume={33},
			pages={035001},
		}
		
		\bib{wong2014pdo}{book}{
			author={Wong, M.~W.},
			title={An introduction to pseudo-differential operators},
			edition={Third},
			series={Series on Analysis, Applications and Computation},
			publisher={World Scientific Publishing Co. Pte. Ltd., Hackensack, NJ},
			date={2014},
			volume={6},
			ISBN={978-981-4583-08-4},
			url={https://doi.org/10.1142/9074},
			review={\MR{3222682}},
		}
		
		\bib{Yuan1}{article}{
			author={Yuan, Ganghua},
			title={Determination of two kinds of sources simultaneously for a
				stochastic wave equation},
			date={2015},
			ISSN={0266-5611},
			journal={Inverse Probl.},
			volume={31},
			number={8},
			pages={085003, 13},
			url={https://doi.org/10.1088/0266-5611/31/8/085003},
			review={\MR{3377100}},
		}
		
		\bib{Zhang2015}{article}{
			author={Zhang, Deyue},
			author={Guo, Yukun},
			title={Fourier method for solving the multi-frequency inverse source
				problem for the {H}elmholtz equation},
			date={2015},
			ISSN={0266-5611},
			journal={Inverse Probl.},
			volume={31},
			number={3},
			pages={035007, 30},
			url={https://doi.org/10.1088/0266-5611/31/3/035007},
			review={\MR{3319373}},
		}
		
		\bib{zw2012semi}{book}{
			author={Zworski, Maciej},
			title={Semiclassical analysis},
			series={Graduate Studies in Mathematics},
			publisher={American Mathematical Society, Providence, RI},
			date={2012},
			volume={138},
			ISBN={978-0-8218-8320-4},
			url={https://doi.org/10.1090/gsm/138},
			review={\MR{2952218}},
		}
		
	\end{biblist}
\end{bibdiv}

{


}


\end{document}